\documentclass[12pt]{amsart}
\usepackage{fullpage}
\usepackage{amsmath,amsthm}
\usepackage{comment}
\usepackage{amssymb}
\usepackage{graphicx,setspace,amscd,float,hyperref,enumitem}
\usepackage{xypic}
\usepackage{multicol}
\usepackage{enumitem}
\usepackage{courier}

\newtheorem*{mainthmA}{Theorem A}
\newtheorem*{mainthmB}{Theorem B}

\newtheorem*{corC}{Corollary C}

\newtheorem{thm}{Theorem}[section]
\newtheorem{lem}[thm]{Lemma}
\newtheorem{qst}[thm]{Question}

\newtheorem{prop}[thm]{Proposition}
\newtheorem{cor}[thm]{Corollary}
\theoremstyle{definition}
\newtheorem{df}[thm]{Definition}
\newtheorem{rk}[thm]{Remark}
\newtheorem{ex}[thm]{Example}

\newtheorem{nt}[thm]{Notation}

\newcommand{\RR}{\mathbb R}
\newcommand{\NN}{\mathbb N}

\newcommand{\G}{\Gamma}

\newcommand{\ol}{\overline}
\newcommand{\mG}{\mathcal{G}}
\newcommand{\mL}{\mathcal{L}}

\newcommand{\mF}{\mathcal{F}}

\newcommand{\mT}{\mathcal{T}}
\newcommand{\mP}{\mathcal{P}}
\newcommand{\mS}{\mathcal{S}}

\newcommand{\vphi}{\varphi}
\newcommand{\veps}{\varepsilon}

  \newcommand{\Z}{\mathbb Z}

\newcommand{\teich}{Teichm\"{u}ller }

\newcommand{\pgraph}{\Delta_r}
\newcommand{\periodiclines}{\mathcal{A}_r}

\newcommand{\from}{\colon}

\newcommand{\al}{\alpha}
\newcommand{\out}{\textup{Out}(F_r)}

\newcommand{\os}{CV_r}

\newcommand{\sig}{\sigma}
\newcommand{\lam}{\lambda}

\newcommand{\uos}{\widehat{CV}_r}
  \renewcommand{\SS}{\mathbb{S}}

\begin{document}

\title{Stable Strata of Geodesics in Outer Space}
\author{Yael Algom-Kfir, Ilya Kapovich, Catherine Pfaff}

\address{\tt Department of Mathematics, University of Haifa \newline
  \indent Mount Carmel;  Haifa, 31905;  Israel
  \newline \indent  {\url{http://www.math.haifa.ac.il/algomkfir/}}, } \email{\tt yalgom@univ.haifa.ac.il}

\address{\tt  Department of Mathematics, University of Illinois at Urbana-Champaign\newline
  \indent 1409 West Green Street, Urbana, IL 61801
  \newline \indent  http://www.math.uiuc.edu/\~{}kapovich, } \email{\tt kapovich@math.uiuc.edu}

\address{\tt Department of Mathematics, University of California at Santa Barbara \newline
  \indent South Hall, Room 6607; Santa Barbara, CA 93106-3080
  \newline
  \indent  {\url{http://math.ucsb.edu/~cpfaff/}}, } \email{\tt cpfaff@math.ucsb.edu}

\date{}

\thanks{The authors acknowledge support from U.S. National Science Foundation grants DMS 1107452, 1107263, 1107367 ``RNMS: Geometric structures And Representation varieties'' (the GEAR Network). The first author was funded by ISF grant 1941/14. The second author was supported by the NSF grants DMS-1405146 and DMS-1710868. All three authors acknowledge the support of the Mathematical Sciences Research Institute during the Fall 2017 semester.}

\maketitle

\begin{abstract}
In this paper we propose an Outer space analogue for the principal stratum of the unit tangent bundle to the \teich space $\mT(S)$ of a closed hyperbolic surface $S$. More specifically, we focus on properties of the geodesics in \teich space determined by the principal stratum. We show that the analogous Outer space ``principal" periodic geodesics share certain stability properties with the principal stratum geodesics of \teich space. We also show that the stratification of  periodic geodesics in Outer space exhibits some new pathological phenomena not present in the Teichm\"uller space context. \end{abstract}

\section{Introduction}
Let $S$ be a closed oriented surface of genus $\ge 2$ and let $\mathcal T(S)$ be the Teichm\"uller space of $S$.  Recall that the unit (co)tangent bundle to $\mathcal T(S)$ is canonically identified with the space $\mathcal Q^1(S)$ of unit area holomorphic quadratic differentials on $S$. The space $\mathcal Q^1(S)$ has a natural stratification, invariant under the Teichm\"uller geodesic flow, according to the orders of zeros of a quadratic differential, with the \emph{principal stratum} $\mathcal Q_{princ}^1(S)$ consisting of quadratic differentials where all zeros are simple.  This stratification of $\mathcal Q^1(S)$  defines the corresponding stratification of the space $\mathbb G\mathcal T(S)$ of all bi-infinite directed \teich geodesics in $\mathcal T(S)$, by looking at the unit tangent vector to $L\in \mathbb G\mathcal T(S)$  at some (equivalently, any) point $X\in L$. Thus we also get a subset $\mathbb G_{princ}\mathcal T(S)\subseteq \mathbb G\mathcal T(S)$ consisting of all \teich geodesics $L\in \mathbb G\mathcal T(S)$ with defining tangent vectors in $\mathcal Q_{princ}^1(S)$.

By the classic work of Thurston, the space $PMF(S)$  of projective measured foliations on $S$ can be viewed as the Thurston boundary $\partial \mathcal T(S)$ of $\mathcal T(S)$.
Note that there is no canonical stratification of $\partial \mathcal T(S)$ corresponding to the stratification of $\mathcal Q^1(S)$ discussed above.  Indeed, one can show that for every point $[\mu]\in \partial \mathcal T(S)$ there exists some \teich geodesic ray $\rho\subseteq \mathcal T(S)$ such that the initial tangent vector of $\rho$ belongs to $\mathcal Q_{princ}^1(S)$. However, for a pseudo-Anosov $g\in MCG(S)$ there is a well-defined notion of $g$ being \emph{principal}, which corresponds to the bi-infinite $g$-periodic geodesic in $\mathcal T(S)$ being defined by a tangent vector from $\mathcal Q_{princ}^1(S)$, or equivalently, to the stable foliation $[\mu_+(g)]$ coming from a quadratic differential in the principal stratum $\mathcal Q_{princ}^1(S)$.

An important result of Kaimanovich and Masur~\cite{km96} concerns the boundary behavior of a random walk on the mapping class group $MCG(S)$, satisfying some mild restrictions. They proved that for every basepoint $X\in\mathcal T(S)$, for almost every trajectory $\omega$ of this random walk, projecting $\omega$ to $\mathcal T(S)$ from the basepoint $X$ via the orbit map gives a sequence in $\mathcal T(S)$ that converges to a uniquely ergodic point of $\partial \mathcal T(S)=PMF(S)$.  Thus we get an exit measure $\nu_X$ on $\partial \mathcal T(S)$ for the projected random walk in $\mathcal T(S)$ starting at $X$. This exit measure is supported on the set $\mathcal{UE}\subseteq PMF(S)$ of uniquely ergodic projective measured foliations.  Maher~\cite{m11} later showed that, again under some mild assumptions on the random walk, the element $g_n\in MCG(S)$, obtained after $n$ steps of the walk, is pseudo-Anosov with probability tending to $1$ and $n\to\infty$. (Rivin~\cite{riv08} had earlier proved the same conclusion about $g_n$ for the simple random walk on $MCG(S)$ corresponding to a finite generating set of $MCG(S)$.)
However, until recently, little else has been known about the properties of a ``random" point of $\partial \mathcal T(S)$ corresponding to the exit measure $\nu_X$, or about the properties of the stable foliation $[\mu_+(g_n)]$ of the pseudo-Anosov $g_n$ as above.

In \cite{gm16}, Gadre and Maher shed light on these questions. They proved that if the support of a random walk on $MCG(S)$ is ``sufficiently large" and contains a principal pseudo-Anosov $g$, then  for every $X\in \mathcal T(X)$ and for $\nu_X$-a.e. point $\xi\in \mathcal UE\subseteq \partial \mathcal T(S)$, the \teich geodesic from $X$ to $\xi$ has its initial tangent vector in $\mathcal Q_{princ}^1(S)$. They also proved that in this setting, with probability tending to $1$ as $n\to\infty$, after $n$ steps the random walk produces a principal pseudo-Anosov $g_n\in MCG(S)$.
Gadre and Maher also obtained the following stability result for principal axes.
For $X,Y$ in the axis $L_g\subseteq \mathcal T(S)$ of $g$ and for $R\ge 0$, denote by $\Gamma_R(X,Y)$ the collection of all oriented bi-infinite \teich geodesics $L\subseteq \mathcal T(S)$ with uniquely ergodic vertical and horizontal foliations such that $B(X,R)\cap L \neq \emptyset$, $B(Y,R)\cap L\ne \emptyset$ and such that the first point in $L\cap (B(X,R)\cup B(Y,R))$ belongs to $B(X,R)$. Here the balls $B(X,R)$ and $B(Y,R)$ are taken with respect to the \teich metric on $\mathcal T(S)$. A crucial ingredient in the proof of the main results of \cite{gm16} is the following ``stability" property of $L_g$ for a principal pseudo-Anosov $g$; see \cite[Proposition~2.7]{gm16}:

\begin{thm}[Gadre-Maher]\label{thm:GM}
Let $S$ be a closed oriented surface of genus $\ge 2$ and let $g\in MCG(S)$ be a principal pseudo-Anosov. Then for any $R\ge 0$ there exists $D>0$ such that if $X,Y\in L_g$ have $d(X,Y)\ge D$, then every $L\in \Gamma_R(X,Y)$ is principal.
\end{thm}

We are interested in investigating the corresponding questions in the $Out(F_r)$ setting, where $r\ge 2$. Similar to the mapping class group setting, it is by now well-known~\cite{mt14,tt16} that, under some mild assumption on the support, for a random walk on $Out(F_r)$, an element $\vphi_n\in Out(F_r)$ obtained after $n$ steps of the walk is atoroidal fully irreducible with probability tending to $1$ as $n\to\infty$. It is also known that projecting a random orbit of this walk to the Culler-Vogtmann Outer space $\os$ (starting at some basepoint $X\in \os$) gives a sequence in $\os$ that with probability $1$ converges to some point in $[T]\in \partial \os$ and, moreover, that the $\mathbb R$-tree $T$ is uniquely ergodic~\cite{npr14}. The proofs that $\vphi_n$ is generically fully irreducible and atoroidal involve projecting a random walk on $Out(F_r)$ to the free factor complex in the first case and to the co-surface graph in the second case. These are both Gromov hyperbolic and one argues that $\vphi_n$ acts loxodromically on the hyperbolic graph in question. Addressing the index properties,  Kapovich and Pfaff~\cite{kp15} proved that for a  ``train-track directed" random walk on $Out(F_r)$, the element $\vphi_n$ is, with an asymptotically positive probability, an ageometric fully irreducible outer automorphism with a 1-element index list $\{\frac{3}{2}-r\}$ and that the corresponding  ideal Whitehead graph is complete (the relevant definitions are discussed next below).
We wish to understand how this statement generalizes to the case of a more general random walk on $\out$.

One of the difficulties in the $Out(F_r)$ setting is finding a suitable notion of a ``principal stratum." In the original context of a closed hyperbolic surface $X$, if $[\mu]\in PMF(S)$ is uniquely ergodic and with the dual $\mathbb R$-tree having all branch-points being trivalent, then $[\mu]\in \mathcal Q_{princ}$. This fact motivates us to use the index properties of a geodesic in Outer space when defining strata in the space of such geodesics.

Given a nongeometric fully irreducible $\vphi\in Out(F_r)$ (where ``nongeometric'' means that $\vphi$ is not induced by a homeomorphism of a surface $S$ with $\pi_1(S)\cong F_r$) one can define a conjugacy class invariant called the \emph{ideal Whitehead graph} $IW(\vphi)$ of $\vphi$ (see \S 2.6 for details). The graph $IW(\vphi)$ captures essential information about the structure of the attracting lamination of $\vphi$ and therefore of branch-points of the stable $\mathbb R$-tree $T_\vphi$ of $\vphi$ (as well as about the interaction of ``directions'' in $T_\vphi$ at those branch-points). The graph $IW(\vphi)$ can be read-off, via an explicit procedure, from any train track  representative of $\vphi$. In addition, one can also define the \emph{index list} for $\vphi$ (recording the sizes of components of $IW(\vphi)$), and the \emph{index sum} $i(\vphi)$, obtained by summing up the numbers in the index list of $\vphi$.
Unlike in the surface case, there may be many types of Ideal Whitehead graphs with the same index list and we shall see that stability properties are related to the graph types rather than the index lists (or sums).
We say that a finite graph $\mathcal G$ is $r$-\emph{dominant} if $\mathcal G$ is a union of complete graphs, each with $\ge 3$ vertices, and if the index sum of $\mathcal G$ is $\frac{3}{2}-r$. Of special interest is the $r$-dominant graph all of whose components are triangles, we denote it $\pgraph$.

Let $r\ge 3$ and let $\os$ denote the (projectivized) Culler-Vogtmann Outer space for the free group $F_r$. We denote by $\mathcal F_r$ the set of all bi-infinite fold lines in $\os$, where folds are performed one at a time (see Definition \ref{d:SimpleFoldLines} below for the precise formulation).  Note that all elements of $\mathcal F_r$ are bi-infinite geodesics for the asymmetric Lipschitz metric on $\os$.
We denote by $\periodiclines$ the set of all ``axes,'' i.e. the set of all $L\in \mathcal F_r$ such that $L$ is a periodic fold line defined by an expanding irreducible train track representative of some $\vphi\in Out(F_r)$. We endow $\mathcal F_r$ with a natural topology, where for $L, L'\in \mathcal F_r$, the line $L'$ is ``close'' to $L$ if there exist a ``large'' $R\ge 1$ and  a ``small'' $\veps>0$ such that some subsegment $J$ of $L'$ of length $R$ is contained in the $\veps$-neighborhood of $L$ (with respect to the symmetrized Lipschitz metric on $\os$). See Definition~\ref{d:TopologyOnSpaceFoldLines} below for details. All of the various subsets of $\mathcal F_r$ discussed below are then given the subspace topology.

\begin{df}[Dominant and principal strata, and their basins]\label{d:strataNbasins}
Let $r\ge 3$ and let $\mathcal G$ be a graph. We define the \emph{$\mathcal G$-basin} $B\mathcal S_r(\mathcal G)\subseteq \periodiclines$ as the set of all $L\in \periodiclines$ such that $L$ is $\vphi$-periodic for $\vphi\in Out(F_r)$ ageometric fully irreducible and satisfying that $IW(\vphi)$ is a union of components of $\mathcal G$. We define the \emph{$\mathcal G$-stratum} $\mathcal S_r(\mathcal G)\subseteq B\mathcal S_r(\mathcal G)$ as those lines for which the corresponding $\vphi$ satisfies $IW(\vphi) \cong \mG$. Thus $\mathcal S_r(\mathcal G)\subseteq B\mathcal P_r(\mathcal G) \subseteq \periodiclines$.

If $\mathcal G$ is an $r$-dominant (resp. $r$-principal) graph, we say $\mathcal S_r(\mathcal G)$ is a \emph{dominant stratum}  (resp. $\mP_r$ is the $r$-principal stratum) and $B\mS_r(\mG)$ is dominant (resp. $B\mP_r$ is the principal basin).
\end{df}

The results of Mosher-Pfaff~\cite{mp13} imply that if $\vphi\in Out(F_r)$ is $\mathcal G$-dominant for some $r$-dominant graph $\mathcal G$, then $\vphi$ is a ``lone axis" fully irreducible outer automorphism, i.e. $\vphi$ has a unique axis in $\os$. In particular, this fact applies to all principal $\vphi\in Out(F_r)$.

Recall that $\os$ is a simplicial complex of dimension $3r-4$ (with some faces missing).  For an integer $k\ge 0$, we will denote by $\os^{(k)}$ the $k$-skeleton of $\os$.
Our main result is the following attracting/stability property for dominant strata:

\begin{mainthmA}\label{t:A}
Let $r\ge 3$ and let $\mathcal G$ be an $r$-dominant graph. Let $L \in \mathcal S_r(\mathcal G)$. Then there exist $0\le k\le 3r-4$ and a neighborhood $U\subseteq \periodiclines$ of $L$ in $\periodiclines$ with the following properties:

\begin{itemize}
\item[(a)]For each $L'\in U$ with $L'\subseteq \os^{(k)}$, we have $L'\in B\mathcal S_r(\mathcal G)$.
\item[(b)] For each $L'\in U$ with $L'\subseteq \os^{(k)}$ and with $L'$ containing no full folds, we have $L'\in \mathcal S_r(\mathcal G)$.
\end{itemize}
\end{mainthmA}
See Definition~\ref{d:Folds} for terminology regarding full folds. Note, the conclusion of Theorem~A implies that each $L'\in U$ is an axis of an ageometric fully irreducible element of $Out(F_r)$. Moreover, in the case of (b), $L'$ is the unique axis of that fully irreducible in $\os$ (\cite{mp13}).

Our results suggest that for a reasonable random walk on $Out(F_r)$, for a random fully irreducible $\vphi_n\in Out(F_r)$ obtained after $n$ steps of the walk, there are several possibilities for $IW(\vphi_n)$ that each occur with an asymptotically positive probability as $n\to\infty$.

\begin{qst}
Is the conclusion of Theorem A true only for an $r$-dominant $\mG$?
\end{qst}

It turns out that it is, in general, not possible to replace $B\mathcal S_r(\mathcal G)$ by $\mathcal S_r(\mathcal G)$ in the conclusion of Theorem~A(a) above. We show that certain kinds of pathologies exist that can force $L'\in U$ to fall out of the dominant  $\mathcal G$-stratum and that the best one can conclude is that $L'\in B\mathcal S_r(\mathcal G)$:

\begin{mainthmB}\label{t:B}
There exists a principal fully irreducible outer automorphism $\vphi \in Out(F_3)$ with a train track representative $f \colon \Gamma \to\Gamma$ with a Stallings fold decomposition $\mathfrak f$, such that for every $n\ge 1$ there exists a nonprincipal fully irreducible outer automorphism  $\psi_n \in Out(F_3)$ with a train track representative $g_n \colon \Gamma \to\Gamma$ with a Stallings fold decomposition $\mathfrak g_n$ such that $\mathfrak g_n$ starts with $\mathfrak f^{n}$.
\end{mainthmB}

Theorem~B immediately implies:

\begin{corC}
For $r=3$, there exist a principal periodic geodesic $L\in \mP_r$ in $\os$ and a sequence of nonprincipal periodic geodesics $\{L_n\}_{n=1}^\infty \subseteq B\mP_r - \mP_r$ such that $\displaystyle \lim_{n \to \infty} L_n = L$.
\end{corC}
The cause of the pathologies exhibited in Theorem~B is that the folding process may identify vertices. Hence, some $f$-periodic vertices of $\Gamma$ may become nonperiodic for $g_n$. Such vertices contribute to $IW(\vphi)$ but not to $IW(\psi_n)$.

\subsection*{Acknowledgements}
This paper arose in response to a question of Lee Mosher. The authors would like to thank Mladen Bestvina, Joseph Maher, Lee Mosher, and Kasra Rafi for helpful and interesting conversations, as well as the MSRI for its hospitality.

\section{Background \& Definitions}{\label{ss:PrelimDfns}}

Given a free group $F_r$ of rank $r \geq 2$, we choose once and for all a free basis $A=\{X_1, \dots, X_r\}$. Let $R_r = \vee_{i=1}^r \SS^1$ denote the graph with one vertex and $r$ edges.
We choose also once and for all an orientation on $R_r$ and
identify each positive edge of $R_r$ with an element of the chosen free basis. Thus, a cyclically reduced word in the basis corresponds to an immersed loop in $R_r$.

\subsection{Outer space $\os$}{\label{ss:CVr}}

\begin{df}[Marked metric $F_r$-graph]{\label{d:markedmetricgraph}}
Let $r\ge 2$ be an integer. A \emph{marked metric graph} for $F_r$ is a triple $(\Gamma, m, \ell)$ which satisfies:
\begin{itemize}
\item $\Gamma$ is a finite 1-dimensional CW complex, with the 0-cells \emph{vertices} and the 1-cells \emph{edges}.
\item For each vertex $v$, $deg(v) \geq 3$.
\item Each open 1-cell $e$ of $\Gamma$ is given a positive \emph{length} $L(e)>0$, and $\Gamma$ is endowed with a metric $\ell$ such that for each open 1-cell $e$ of $\Gamma$ there is a locally isometric bijection between $e$ and the interval $(0,L(e))\subseteq \mathbb R$.
\item $m$ is a homotopy equivalence $m \from R_r \to \Gamma$, which we call a \emph{marking}.
\end{itemize}
\end{df}

\begin{nt} We use the following notation.
\begin{enumerate}
\item We sometimes write $(\Gamma, m)$ for $(\Gamma, m, \ell)$ if the metric is otherwise clear or irrelevant.
\item Given a graph $\Gamma$ (metric or topological), we let $E(\Gamma)$ denote the set of oriented edges of $\Gamma$ and let $V(\Gamma)$ denote the vertex set of $\Gamma$.
\end{enumerate}
\end{nt}

\begin{df}[Change of marking \& marked graph equivalence]
Let $(\Gamma, m)$ and $(\Gamma', m')$ be marked graphs. Then a \emph{change of marking} is a continuous map $f \from \Gamma \to \Gamma'$ so that $m'$ is homotopic to $f \circ m$.
Two $F_r$-marked (metric) graphs  $(\Gamma, m) $ and $(\Gamma', m')$ are \emph{equivalent} if there exists an isometric change of marking $\varphi \colon \Gamma \to \Gamma'$.
\end{df}

\begin{df}[Unprojectivized Outer space]{\label{d:uos}}

The \emph{(rank-$r$) unprojectivized Outer space} $\uos$ is the space of equivalence classes of $F_r$-marked metric graphs. By abuse of notation, we usually still denote the equivalence class of $(\Gamma, m) $ by $(\Gamma, m)$, or of $(\Gamma, m, \ell)$ by $(\Gamma, m, \ell)$.

For a marked metric graph $(\Gamma, m)$ denote by $vol(\Gamma, m, \ell)$, or just $vol(\Gamma)$, the sum of the $\ell$-lengths of the 1-cells in $\Gamma$. Note that $vol(\Gamma, m, \ell)$ is preserved by the above equivalence relation, so that $vol(\Gamma, m)$ is well-defined for points of $\uos$.
\end{df}

\begin{df}[(Projectivized) Outer space]{\label{d:pos}}
Let $r\ge 2$ be an integer.
For each $r \geq 2$ the (\emph{rank-$r$}) \emph{Outer space} $\os$ is the set of $(\Gamma, m, \ell)\in\uos$ with $vol(\Gamma, m, \ell)=1$. There is a map from $q \from \uos \to \os$ normalizing the graph volume, i.e. if $(\Gamma, m, \ell)$ is a marked metric graph, then $q(\Gamma, m,\ell) = (\Gamma, m, \frac{1}{vol(\Gamma, m, \ell)} \ell)$.

Note that $\mathbb R_{>0}$ has a natural action on $\uos$ by multiplying the metric on $\Gamma$ by a positive real number. There is a canonical identification between $\os$ and the quotient set $\uos/\mathbb R_{>0}$ and we will usually not distinguish between these two sets.
\end{df}

\begin{df}[Simplicial structure on $CV_r$]\label{simplexDefn}
Let $\Gamma$ be a topological graph and $m \from R_r \to \Gamma$ a homotopy equivalence, so that $(\Gamma, m)$ is a marked graph.
We denote the \emph{simplex} $\sig$ in $\os$ corresponding to $(\Gamma, m)$ by
$$\sig_{(\Gamma, m)} := \{ (\Gamma,m,\ell) \in \os \}.$$
By enumerating $E(\Gamma)$, we can identify $\sig_{(G,\mu)}$ with the open simplex
$$\Sigma_{|E|} = \left\{ \overrightarrow{v} \in \RR_+^{|E|} ~ \left| ~ \sum_{i=1}^{|E|} v_i = 1 \right. \right\}.$$
\end{df}

\begin{df}[$CV^{(k)}_r$]
We let $CV^{(k)}_r$ denote the $k$-skeleton of $\os$.
\end{df}

\begin{df}[Simplicial metric]
Given an open simplex $\sig_{(\Gamma, m)}$ in $\os$, the \emph{simplicial metric} on $\sig_{(\Gamma, m)}$ is the Eucliden metric on $\Sigma_{|E|}$.
We also denote by $d_{simp}$ the extension of this metric to a path metric on $\os$.
(There is another (asymmetric) metric on Outer space see Definition \ref{dfLipMetric}).
\end{df}

\begin{df}[Topology on $\uos$]
We call the full preimage under $q$ (see Definition \ref{d:uos}) of a simplex in $\os$ an \emph{unprojectivized simplex} in $\uos$. The unprojectivized Outer space $\uos$ is topologized by giving it the the structure of an ideal simplicial complex built from (unprojectivized) open simplices (see \cite{v02} for details). Faces of $\sig_{(\Gamma, m)}$ arise by letting the edges of a tree in $\Gamma$ have length 0. The projectivized outer space $\os\subseteq\uos$ is given the subspace topology, and can also be thought of as an ideal simplicial complex built from open simplices. The subspace topology on $\os$ coincides with the quotient topology on $\uos/\mathbb R_{>0}$.
\end{df}

\subsection{Train track maps \& gate structures}{\label{ss:TTs}}

\begin{df}[Graph maps \& train track maps]\label{d:Maps}
We call a continuous map of graphs $g \colon \Gamma \to \Gamma'$ a \emph{graph map} if it takes vertices to vertices and is locally injective on the interior of each edge. A self graph map $g \from \Gamma \to \Gamma$ is a \emph{train track map} if $g$ is a homotopy equivalence and if for each $k\ge 1$ the map $g^k$ is locally injective on edge interiors.

We call the train track map $g$ \emph{expanding} if for each edge $e \in E(\Gamma)$ we have that $|g^n(e)|\to\infty$ as $n\to\infty$, where for a path $\gamma$ we use $|\gamma|$ to denote the number of edges $\gamma$ traverses (with multiplicity).
\end{df}

\begin{df}[Directions]\label{d:Directions} For each $x\in \Gamma$ we let $\mathcal{D}(x)$ denote the set of \emph{directions} at $x$, i.e. germs of initial segments of edges emanating from $x$. For each edge $e \in E(\Gamma)$, we let $D(e)$ denote the initial direction of $e$. For an edge-path $\gamma=e_1 \dots e_k$, we let $D \gamma = D(e_1)$. Let $g \from \Gamma \to \Gamma'$ be a graph map. We denote by \emph{$Dg$} the map of directions induced by $g$, i.e. $Dg(d)=D(g(e))$ for $d=D(e)$. For a self-map, i.e. one where $\Gamma =\Gamma'$, a direction $d$ is \emph{periodic} if $Dg^k(d)=d$ for some $k>0$ and \emph{fixed} when $k=1$.
\end{df}

\begin{df}[Turns \& gates]\label{d:GateStructures} Let $g \colon \Gamma \to \Gamma'$ be a graph map. We call an unordered pair of directions $\{d_i, d_j\}$ a \emph{turn}, and a \emph{degenerate turn} if $d_i = d_j$.
We denote by $Tg$ the map induced by $Dg$ on the turns of $\Gamma$.
A turn $\tau$ is called g-\emph{prenull} if $Tg(\tau)$ is degenerate.
When $g\from \Gamma \to \Gamma$ is a self-map, the turn $\tau$ is called an \emph{illegal turn} for $g$ if $Tg^k(\tau)$ is degenerate for some $k$ and a \emph{legal turn} otherwise. We call a $g$ \emph{transparent} if each illegal turn is prenull. Notice that every graph self-map has a transparent power.

Considering the directions of an illegal turn equivalent, one can define an equivalence relation on the set of directions at a vertex. We call the equivalence classes \emph{gates} and call the partitioning of the directions at each vertex into gates the \emph{induced gate structure}.

For a path $\gamma=e_1e_2 \dots e_{k-1}e_k$ in $\Gamma$ where $e_1$ and $e_k$ may be partial edges, we say $\gamma$ \emph{takes} $\{\overline{e_i}, e_{i+1}\}$ for each $1 \leq i < k$. For both edges and paths we more generally use an ``overline'' to denote a reversal of orientation. Given a graph map $g \from \Gamma \to \Gamma'$, we say that a turn $T$ in $\Gamma'$ is \emph{$g$-taken} if there exists an edge $e$ so that $g(e)$ takes $T$. A path $\gamma$ is \emph{legal} with respect to a train track structure on $\Gamma$ if $\gamma$ only takes turns that are legal in this train track structure.
\end{df}

\begin{df}[Irreducible \& fully irreducible] We call a train track map \emph{irreducible} if it has no proper invariant subgraph with a noncontractible component.

An outer automorphism $\vphi\in\out$ is \emph{fully irreducible} if no positive power preserves the conjugacy class of a proper free factor of $F_r$. Bestvina and Handel \cite{bh92} proved that every (fully) irreducible outer automorphism admits an irreducible train track representatives.
\end{df}

\begin{df}[Transition matrix, Perron-Frobenius matrix, Perron-Frobenius eigenvalue] The \emph{transition matrix} of a train track map $g \from \Gamma \to \Gamma$ is the square $|E(\Gamma)| \times |E(\Gamma)|$ matrix $(a_{ij})$ such that $a_{ij}$, for each $i$ and $j$, is the number of times $g(e_i)$ passes over $e_j$ in either direction.
A transition matrix $A=[a_{ij}]$ is \emph{Perron-Frobenius (PF)} if there exists an $N$ such that, for all $k \geq N$, $A^k$ is strictly positive.
By Perron-Frobenius theory, we know that each such matrix has a unique eigenvalue of maximal modulus and that this eigenvalue is real.
This eigenvalue is called the \emph{Perron-Frobenius (PF) eigenvalue} of $A$.
\end{df}

\subsection{Fold lines}{\label{ss:FoldLines}}

\begin{df}[Fold lines]{\label{d:FoldLines}}
A \emph{fold line} in $\uos$ is a continuous, injective, proper function $\mathbb{R} \to \uos$ defined by a continuous 1-parameter family of marked graphs $t \to \Gamma_t$ and a family of differences of markings $\hat h_{st} \colon \hat\Gamma_t \to \hat\Gamma_s$ defined for $t \leq s \in \mathbb{R}$, satisfying:
\begin{enumerate}
\item  $\hat h_{ts}$ is a local isometry on each edge for all $s \leq t \in \mathbb{R}$.
\item $\hat h_{us} \circ \hat h_{st} = \hat h_{ut}$ for all $t \leq s \leq u \in \mathbb{R}$ and $\hat h_{ss} \colon \Gamma_s \to \Gamma_s$ is the identity for all $s \in \mathbb{R}$.
\end{enumerate}
A fold line in $\os$ is the $q$-image (where $q$ is the normalizing map, see Definition \ref{d:pos}) of a fold line. We shall denote $q(\hat \Gamma_t)$ and $q\circ \hat h_{s,t}$ by $\Gamma_t$ and $h_{s,t}$ respectively.
\end{df}

\begin{df}[Simple fold lines]{\label{d:SimpleFoldLines}}
A fold line in Outer space $\mathbb{R} \to \os$ is said to be \emph{simple} if there exists a subdivision of $\mathbb R$ by points $(t_i)_{i\in \mathbb Z}$
\[
  \dots t_{i-1}< t_i <t_{i+1} \dots
\]
such that $\lim_{i\to\infty} {t_i}=\infty$, $\lim_{i\to-\infty} t_i=-\infty$ and such that the following holds:

For each $i\in \mathbb Z$ there exist distinct edges $e,e'$ in $\Gamma_{t_i}$, with a common initial vertex, such that: For each $s\in (t_i,t_{i+1}]$ the map $h_{s t_i}\colon \Gamma_{t_i}\to \Gamma_s$ identifies an initial segment of $e$ with an initial segment of $e'$, with no other identifications (that is, $h_{s t_i}$ is injective on the complement of those two initial segments in $\Gamma_{t_i}$).
\end{df}

\begin{rk}[Simple fold lines]{\label{r:SimpleFoldLines}}
All fold lines that we consider in this paper will be simple.
\end{rk}

\begin{df}[Stallings folds]{\label{d:Folds}}
Stallings introduced ``folds'' in \cite{s83}. Let $g \colon \Gamma \to \Gamma'$ be a homotopy equivalence of marked graphs. Let $e_1' \subset e_1$ and $e_2' \subset e_2$ be maximal, initial, nontrivial subsegments of edges $e_1$ and $e_2$ emanating from a common vertex and satisfying that $g(e_1')=g(e_2')$ as edge paths and that the terminal endpoints of $e_1'$ and $e_2'$ are distinct points in $g^{-1}(\mathcal{V}(\Gamma))$. Redefine $\Gamma$ to have vertices at the endpoints of $e_1'$ and $e_2'$ if necessary. One can obtain a graph $\Gamma_1$ by identifying the points of $e_1'$ and $e_2'$ that have the same image under $g$, a process we will call \emph{folding}.

Let $\mF$ be a fold of $e_1$ and $e_2$. We call $\mF$ a \emph{full fold} if the entirety of $e_1$ and $e_2$ are identified. We call $\mF$ a \emph{proper full fold} if only an initial subsegment of one of $e_1$ or $e_2$ is folded with the entirety of the other. We call $\mF$ a \emph{partial fold} if neither $e_1$ nor $e_2$ is entirely folded.
\end{df}

\begin{df}[Stallings fold decomposition]{\label{d:StallingsFoldDecomposition}}
Stallings \cite{s83} also showed that if $g \colon \Gamma \to \Gamma'$ is a homotopy equivalence graph map, then $g$ factors as a composition of folds and a final homeomorphism. We call such a decomposition a \emph{Stallings fold decomposition}. It can be obtained as follows: At an illegal turn for $g\colon \Gamma  \to \Gamma'$, one can fold two maximal initial segments having the same image in $\Gamma'$ to obtain a map $\mathfrak{g}_1 \colon \Gamma_1 \to \Gamma'$ of the quotient graph $\Gamma_1$. The process can be repeated for $\mathfrak{g}_1$ and recursively. If some $\mathfrak{g}_k \colon \Gamma_{k-1} \to \Gamma$ has no illegal turn, then $\mathfrak{g}_k$ will be a homeomorphism and the fold sequence is complete.
\end{df}

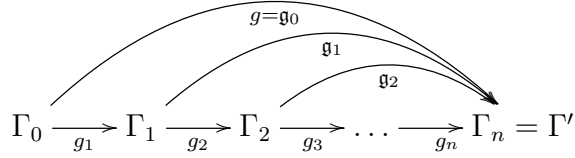
\begin{figure}[ht!]
\[
\xymatrix{\Gamma_0 \ar[r]_{g_1} \ar@/^4pc/[rrrr]_{g=\mathfrak{g}_0} & \Gamma_1 \ar[r]_{g_2} \ar@/^3pc/[rrr]_{\mathfrak{g}_1} & \Gamma_2 \ar[r]_{g_3} \ar@/^2pc/[rr]_{\mathfrak{g}_2} & \dots \ar[r]_{g_n}  & \Gamma_n=\Gamma' \\}
\]
\caption{Constructing a Stallings folds decomposition}
\end{figure}

Notice that choices of illegal turns are made in this process and that different choices lead to different Stallings fold decompositions of the same homotopy equivalence.

When $\Gamma$ is a marked metric graph (of volume 1), we obtain an induced metric on each $\Gamma_k$, which we may renormalize to be again of volume 1.

In \cite{s89}, Skora interpreted a Stallings fold decomposition for a graph map  homotopy equivalence $g\colon \Gamma \to \Gamma'$ as a sequence of folds performed continuously. Let $g\colon \Gamma \to \Gamma$ be an irreducible train track map representing an outer automorphism $\vphi \in Out(F_r)$ and let $\lambda$ be its Perron-Frobenius eigenvalue. Repeating a Stallings fold decomposition of $g$ defines a periodic fold line in Outer space. The discretization of this fold line is depicted in Equation \ref{E:PeriodicFoldLines} below, where it should be noted that $\Gamma_{nK}=\frac{1}{\lambda^n}\Gamma_0 \cdot \varphi^n$, for each integer $n$.

\begin{equation}\label{E:PeriodicFoldLines}
\dots \xrightarrow{} \Gamma_0 \xrightarrow{g_1} \Gamma_1 \xrightarrow{g_2} \cdots \xrightarrow{g_K} \Gamma_K \xrightarrow{g_{K+1}} \Gamma_{K+1} \xrightarrow{g_{K+2}} \cdots \xrightarrow{g_{2K}} \Gamma_{2K} \xrightarrow{g_{2K+1}} \dots
\end{equation}

\begin{df}[Periodic fold lines]\label{d:PeriodicFoldLine}
Let $g\colon \Gamma \to \Gamma$ be an expanding irreducible train track map representing an outer automorphism $\vphi \in Out(F_r)$ and let $\lambda>1$ be its Perron-Frobenius eigenvalue. If $g_1, \dots, g_k$ is a Stallings fold sequence for $g$, the process of Skora defines a path $\mL_0 \from [0,\log\lam] \to \os$ so that the union of $\varphi^k$-translates of $\mL_0$ for all $k$ gives the entire fold line $\mL$ determined by $g$, see Definition \ref{d:StallingsFoldDecomposition}. That is, $\mL \from \RR \to \os$ is defined by $\mL(t) = \mL_0(t - \lfloor \frac{t}{\log\lam}\rfloor) \varphi^{\lfloor \frac{t}{\log\lam}\rfloor}$.
$\mL$ is called a \emph{periodic fold line} for $\vphi$ or, if $\vphi$ is fully irreducible, an \emph{axis} for $\vphi$.
\end{df}

\subsection{Geodesics in Outer space}{\label{ss:GOS}}

\begin{df}[Lipschitz metric]\label{dfLipMetric}
Given an ordered pair of points $(X,Y)$ in the Outer space $\os$, the \emph{Lipschitz distance} $d(X,Y)$ from $X=(\Gamma_X, m_X, \ell_X)$ to $Y=(\Gamma_Y, m_Y, \ell_Y)$ is defined as the logarithm of the minimal Lipschitz constant of a Lipschitz difference of markings. (It is known~\cite{fm11} this minimum is in fact realized and that $d(X,Y)\ge 0$, with $d(X,Y)=0$ if and only if $X=Y$ in $\os$.) We sometimes also denote $d(X,Y)$ by $d_L(X,Y)$.
\end{df}

Let $\al$ be an element of $F_r$. We also denote by $\al$ the corresponding loop in the base rose $R_r$. Let $X=(\Gamma,m)$ be a point in Outer space.
Denote by $\al_X$ the immersed loop (unique up to cyclic reparametrization) in $X$ that is freely homotopic to $m(\al)$.

\begin{df}[Witness]
It is proved in \cite{fm11} that for each ordered pair of points $(X,Y)$ in Outer space there exists an element $\al$ of $F_r$ so that $\log \frac{len(\al_Y)}{len(\al_X)} = d(X,Y)$.
We call each such $\al$ a \emph{witness} of $d(X,Y)$ or of the change of marking from $X$ to $Y$.
\end{df}

\begin{df}[Candidate]
Let $X \in \os$. A loop in $X$ whose image is an embedded circle, an embedded figure-8, or an embedded barbell is called a \emph{candidate} of $X$.
\end{df}

\begin{lem}\cite{fm11}{\label{l:candidatewitnesses}}
For each ordered pair of points $(X,Y)$ in Outer space there exists a candidate loop in $X$ that is a witness of $d(X,Y)$.
\end{lem}

\begin{df}[Geodesic]
A map $\mL \from [0,\ell] \to \os$ is a \emph{Lipshitz geodesic} if
\begin{enumerate}
\item for all $s, t, r \in [0, \ell]$ so that $s \leq t \leq r$ we have
\[ d(\mL(s), \mL(r)) = d(\mL(s), \mL(t))+ d(\mL(t), \mL(r)) \text{ and} \]
\item there exists no $X_0 \in \os$ and no nontrivial subinterval $[a,b]$ of $[0, \ell]$ so that $\mL(t) = X_0$ for all $t \in [a,b]$.
\end{enumerate}
\end{df}

\begin{lem}{\label{l:WitnessLoops}}
Let $\mL = \{ \Gamma_t \}_{t=0}^{\infty}$ be a Lipschitz geodesic ray. Then there exists an element $\al \in F_r$ so that, for each $t \geq 0$, if $\al_t$ denotes the immersed loop representing the conjugacy class of $\al$ in $\Gamma_t$, then $\al_t$ is a witness to $d(\Gamma_t, \Gamma_s)$ for each $t \leq s$.
\end{lem}

\begin{proof}
For each $0 \leq t \leq s$ we have $d(\Gamma_0, \Gamma_s) =d(\Gamma_0,\Gamma_t) + d(\Gamma_t,\Gamma_s)$. Let $f_{s,0}, f_{t,0} , f_{s,t}$ be optimal maps, i.e. maps that realize the equality in Definition \ref{dfLipMetric}, and let $\al$ be a witness loop for $d(\Gamma_0, \Gamma_s)$. Then $ \frac{len(\al_s)}{len(\al_0)} = Lip f_{s,0}  = Lip f_{s,t} Lip f_{t,0} \geq \frac{len(\al_s)}{len(\al_t)} \frac{len(\al_t)}{len(\al_0)} $. Thus, all of the inequalities are in fact equalities i.e. $Lip f_{s,t} =\frac{len(\al_s)}{len(\al_t)}$ and $Lip f_{t,0} = \frac{len(\al_t)}{len(\al_0)}$. Hence, $\al_0$ is a witness for $f_{t,0}$ for all $t \leq s$ and $\al_t$ is a witness for $f_{s,t}$.
Moreover, by Lemma \ref{l:candidatewitnesses}, $\al_0$ can be chosen to be a candidate in $\Gamma_0$. Notice that there are only finitely many candidate loops in $\Gamma_0$. Let $\Theta_s$ be the set of candidate loops that are $f_{s,0}$ witnesses. Then the sequence $\{\Theta_s\}_{s=0}^{\infty}$ consists of a non-empty decreasing finite sets of loops, hence stabilizes as $s \to \infty$. We let
$$\Theta_\infty :=\bigcap_{s=0}^{\infty} \Theta_s.$$
Any $\al_0 \in \Theta_\infty$ is a witness for each $f_{s,0}$ with $0 \leq s$ and, by the discussion above, $\al_{t}$ is a witness of $f_{s,t}$ for each $0 \leq t \leq s$.
\end{proof}

\begin{lem}\label{l:stallingsgoedesics}
Let $g$ be an expanding train track representative of $\vphi \in \out$ and $\mL$ the periodic fold line determined by $g$ as in Definition \ref{d:PeriodicFoldLine}. Then $\mL$ is a Lipschitz geodesic.
\end{lem}

\begin{proof}
For each interval $[t,s] \subset \RR$, let $\mL_{t,s}$ denote the restriction of $\mL$ to $[t,s]$. It suffices to show $\mL_{t,s}$ is a geodesic segment for each pair $t,s$ of positive integer multiples of $\log\lam$. For each such $t,s$ the change of marking map from $\mL(t)$ to $\mL(s)$ is $g^k$, where $k = \frac{s-t}{\log\lam}$.
Given $\al \in F_r$, we denote by $\al_t$ the immersed loop in $\mL(t)$ representing $\al$. Let $\beta \in F_r$ satisfy that $\beta_t$ is $g$-legal. Then $h_{u,t}(\beta_t) = \beta_u$ is immersed for each $t \leq u \leq s$, thus it is a witness for $d(\mL(t), \mL(u))$. Moreover, $\beta_u$ is still not $h_{s,u}$-prenull for each $t \leq u \leq s$. Therefore, $\beta_u$ is a witness for $d(\mL(u), \mL(s))$ for all $u \leq s$. Taking the logarithm of both sides of $\frac{\ell_s(\beta_s)}{\ell_t(\beta_t)} = \frac{\ell_u(\beta_u)}{\ell_t(\beta_t)}\frac{\ell_s(\beta_s)}{\ell_u(\beta_u)}$, we see that $\mL$ is a geodesic.
\end{proof}

\subsection{Nielsen paths \& principal points}{\label{ss:NPs}}

\begin{df}[Nielsen paths] Let $g \colon \Gamma \to \Gamma$ be an expanding irreducible train track map. Bestvina and Handel \cite{bh92} define a nontrivial immersed path $\rho$ in $\Gamma$ to be a \emph{periodic Nielsen path (PNP)} if, for some power $R \geq 1$, we have $g^R(\rho) \cong \rho$ rel endpoints (and just a \emph{Nielsen path (NP)} if $R=1$). An NP $\rho$ is called \emph{indivisible} (hence is an ``iNP'') if it cannot be written as $\rho = \gamma_1\gamma_2$, where $\gamma_1$ and $\gamma_2$ are themselves NPs.
\end{df}

\begin{df}[Ageometric]{\label{d:ageometric}} A fully irreducible outer automorphism is called \emph{ageometric} if it has a train track representative with no PNPs.
\end{df}

Bestvina and Handel describe in \cite[Lemma 3.4]{bh92} the structure of iNPs:

\begin{lem}[\cite{bh92}]{\label{l:iNP}}
Let $g \colon \Gamma\to\Gamma$ be an expanding irreducible train track map and $\rho$ an iNP for $g$. Then $\rho=\bar \rho_1\rho_2$, where $\rho_1$ and $\rho_2$ are nontrivial legal paths originating at a common vertex $v$ and such that the turn at $v$ between $\rho_1$ and $\rho_2$ is a nondegenerate illegal turn for $g$.
\end{lem}

\begin{df}[Principal points] Given a train track map $g \colon  \Gamma \to \Gamma$, following \cite{hm11} we call a point \emph{principal} that is either the endpoint of a PNP or is a periodic vertex with $\geq 3$ periodic directions. Thus, in the absence of PNPs, a point is principal if and only if it is a periodic vertex with $\geq 3$ periodic directions
\end{df}

\begin{df}[Rotationless] An expanding irreducible train track map is called \emph{rotationless} if each principal point and periodic direction is fixed and each PNP is of period one. By \cite[Proposition 3.24]{fh11}, one then defines a fully irreducible $\vphi \in Out(F_r)$ to be \emph{rotationless} if some (equivalently, all) of its train track representatives is rotationless.
\end{df}

\subsection{Whitehead graphs}{\label{ss:wg}}

The following definitions are in \cite{hm11} and \cite{mp13}.

\begin{df}[Whitehead graphs \& indices]{\label{d:WhiteheadGraphs}}
Let $g \from \G \to \G$ be a train track map. The \emph{local Whitehead graph}
$LW(v;\Gamma)$ at a point $v \in \Gamma$ has a vertex for each direction at $v$ and an edge connecting the vertices corresponding to a pair of directions $\{d_1,d_2\}$ if the turn $\{d_1,d_2\}$ is $g^k$-taken for some $k\geq 0$.
The \emph{stable Whitehead graph} $SW(v;\Gamma)$ at a principal point $v$ is the subgraph of $LW(v;\Gamma)$ obtained by restricting to the periodic direction vertices.

Let $g \from \G \to \G$ be a PNP-free train track representative of a fully irreducible $\vphi \in Out(F_r)$. Then the \emph{ideal Whitehead graph $IW(\vphi)$ of $\vphi$} is isomorphic to the disjoint union $\bigsqcup \mathcal{SW}(g;v)$ taken over all principal vertices. Justification of this being an outer automorphism invariant can be found in \cite{hm11, p12a}.

Let $\vphi \in Out(F_r)$ be fully irreducible. For each component $C_i$ of $IW(\vphi)$, let $k_i$ denote the number of vertices of $C_i$. Then the \emph{index sum} is defined as
$i(\vphi) := \sum 1-\frac{k_i}{2}$. Since the index sum can be computed as such from the ideal Whitehead graph, we can define an index sum for an ideal Whitehead graph, or in fact any graph. For a graph $\mG$, we write the index sum as $i(\mG)$.
Writing the terms $1-\frac{k_i}{2}$ as a list, we obtain the \emph{index list} for $\vphi$.
\end{df}

\begin{rk} By \cite{gjll}, we know that all fully irreducible $\varphi \in Out(F_r)$ satisfy $0 > i(\varphi) \geq 1-r$.
An ageometric fully irreducible $\varphi \in Out(F_r)$ can be characterized by satisfying $0 > i(\varphi) > 1-r$. The definition we have given for an ideal Whitehead graph only works for ageometric fully irreducibles. However the index sum is always defined from the ideal Whitehead graph as in Definition \ref{d:WhiteheadGraphs} and general definitions of the ideal Whitehead graph can be found in \cite{p12a} or \cite{hm11}.
\end{rk}

A train track map $g$ induces a simplicial (hence continuous) map $Dg \from LW(g,v) \to LW(g,g(v))$ extending the map of vertices defined by the direction map $Dg$.
When $g$ is rotationless and $v$ a principal vertex, the map $Dg \from LW(g,v) \to LW(g,v)$ has image in $SW(g,v)$. Since $Dg$ acts as the identity on $SW(g,v)$, when viewed as a subgraph of $LW(g,v)$, this map is in fact a surjection $Dg \from LW(g,v) \to SW(g,v)$.

\subsection{Full irreducibility criterion.}

The following lemma is essentially \cite[Proposition 4.1]{IWGII}, with the added observation that a fully irreducible outer automorphism with a PNP-free train track representative is in fact ageometric (by definition). \cite{k14} has a related result.

\begin{prop}[\cite{IWGII}](\emph{The Ageometric Full Irreducibility Criterion (FIC)})\label{prop:FIC}
Let $g\colon \Gamma \to \Gamma$ be a PNP-free, irreducible train track representative of $\vphi \in Out(F_r)$. Suppose that the transition matrix for $g$ is Perron-Frobenius and that all the local Whitehead graphs are connected. Then $\vphi$ is an ageometric fully irreducible outer automorphism.
\end{prop}

\subsection{Lone Axis Fully Irreducible Outer Automorphisms}{\label{ss:LoneAxisFullyIrreducibles}}

In \cite{mp13} Mosher and Pfaff defined the property of being a lone axis fully irreducible outer automorphism. In lay terms this means that there is only one fold line in $\os$ that is invariant under $\vphi$.

\begin{thm}[\cite{mp13}]{\label{t:uniqueaxis}} Let  $\varphi \in Out(F_r)$ be an ageometric fully irreducible outer automorphism. Then $\vphi$ is a lone axis fully irreducible if and only if ~\\
\vspace{-5mm}
\begin{enumerate}
\item the rotationless index satisfies $i(\varphi) = \frac{3}{2}-r$ and 
\item no component of the ideal Whitehead graph $\mathcal{IW}(\varphi)$ has a cut vertex.
\end{enumerate}
\end{thm}

\begin{rk}\label{noPNP}
It will be important for our purposes that each train track representative of an ageometric lone axis fully irreducible $\vphi$ is PNP-free (\cite[Lemma 4.4]{mp13}).
\end{rk}

The unique axis is a periodic fold line and one may choose a particularly nice train track representative to generate it (Definition \ref{d:StallingsFoldDecomposition}).

\begin{prop}[\cite{mp13}]\label{P:EveryVertexPrincipal}
Let $\vphi$ be an ageometric lone axis fully irreducible outer automorphism. Then there exists a train track representative $g \from \Gamma \to \Gamma$ of some power $\vphi^R$ of $\vphi$ so that all vertices of $\Gamma$ are principal, and fixed, and all but one direction is fixed.
\end{prop}

\begin{df}[$A_{\vphi}$]{\label{d:axes}} Given a lone axis fully irreducible outer automorphism $\vphi$, we denote its axis by $A_{\vphi}$. In particular, $A_{\vphi}$ will be the periodic fold line determined by any (and every) train track representative of any positive power of $\vphi$.
\end{df}

\section{Stratification of the space of fold lines}{\label{ss:StratificationSpaceOfFoldLines}}

\subsection{The space of fold lines}{\label{ss:SpaceOfFoldLines}}

We fix a rank $r \geq 3$ throughout this section. Notice that the Outer space $\os$ has dimension $0 \leq k \leq 3r-4$.

\begin{df}[$\mF_r$ \& $\periodiclines$]\label{d:SpaceFoldLines} $\mF_r$ will denote the set of all simple fold lines in $\os$ (see Definition \ref{d:SimpleFoldLines}). $\periodiclines \subset \mF_r$ will denote the set of all periodic fold lines in $\os$ (see Definition \ref{d:PeriodicFoldLine}).
\end{df}

\begin{df}[Topology on $\mF_r$]\label{d:TopologyOnSpaceFoldLines} Let $\mL$ be a geodesic in $\os$. We let $N(\mL,\varepsilon)$ denote the $\veps$-neighborhood of $\mL$ in $\os$ with respect to the symmetrized Lipschitz metric on $\os$. We let $B(\mL,R,\veps) \subset \mF_r$ denote the set of all $\mL'\in \mF_r$ such that $\mL'$ has a length-$R$ subsegment $\beta$ with $\beta \subset N(\mL,\varepsilon)$. For each integer $k$ with $0 \leq k \leq 3r-4$, we denote by $B^k(\mL,R,\veps)$ the set of all $\mL'\in B(\mL,R,\veps)$ such that the line $\mL'$ is contained in the $k$-skeleton $CV^{(k)}_r$.

We topologize $\mF_r$ by using, for each $\mL\in \mF_r$, the family of sets $\{B(\mL,R,\veps)\}_{\veps>0, R\ge 1}$ as the basis of neighborhoods of $\mL$ in $\mF_r$.
\end{df}

\begin{rk}\label{r:TopologyOnSpaceFoldLines}
It is a subtle but rather minor point to decide which metric to use in Definition \ref{d:TopologyOnSpaceFoldLines}.
The Lipschitz metric $d_L$ is not symmetric. Define the symmetrized Lipchitz metric as $d_s(X,Y) = d_L(X,Y)+d_L(Y,X)$. Consider the 4 possible topologies arising from the following generating sets: balls in the symmetrized metric, balls in the simplicial metric, ``incoming balls'' in the Lipschitz metric $B_{in}(X,r) = \{ Y \in \os \mid d_L(Y,X) <r \}$, and ``outgoing balls'' in the Lipschitz metric $B_{out}(X,r) = \{ Y \in \os \mid d_L(X,Y) <r \}$. The four topologies on $\os$ coincide \cite{yak}.
However the same is not true for neighborhoods of geodesics. Let $\mL$ be a geodesic and consider $N_{simp}(\mL,\veps) = \{ Y \mid d_{simp}(Y, \mL) < \veps \}$, similarly define $N_{sym}(\mL,\veps)$. Define
$N_{in}(\mL,\veps) = \{ Y \mid d(Y,\mL)<\veps\}$ and
$N_{out}(\mL,\veps) = \{ Y \mid d(\mL,Y)< \veps \}$.
Then sets of the first three types are equivalent, in that for each of the two types and for all $\veps$, one may find an $\veps'$ so that $N(\mL,\veps)$ of one of the types contains $N(\mL,\veps')$ of the other type. The same is not true for $N_{out}(\mL,\veps)$. There exists a geodesic $\mL$ and $\veps>0$ so that for all $\veps'$, $N_{sym}(\mL,\veps)$ does not contain $N_{out}(\mL,\veps')$. The outgoing neighborhoods are ``too big,'' hence we use the others (these geodesics $\mL$ necessarily don't stay in any ``thick part'' of $\os$).
\end{rk}

\begin{rk}
Note, with the topology defined above, the space $\mF_r$ is nonHausdorff: if two distinct fold lines $\mL,\mL'\in \mF_r$ overlap along a common subray, then each neighborhood of $\mL$ in $\mF_r$ contains $\mL'$ and each neighborhood of $\mL'$ contains $\mL$. Nevertheless, the topology on $\mF_r$ given in Definition~\ref{d:TopologyOnSpaceFoldLines} is natural for our purposes. Moreover, the topology is better behaved when restricted to the subspace $\periodiclines\subseteq \mF_r$, the main object of interest in this paper.
\end{rk}

\subsection{Strata of geodesics}{\label{ss:PrincipalStrata}}

\begin{df}[$r$-Dominant graph]\label{d:DominantStratum}
Let $r \geq 3$ be fixed. A finite graph $\mG$ is \emph{$r$-dominant} if
it is a disjoint union of complete graphs and has index sum $\frac{3}{2}-r$, see Definition \ref{d:WhiteheadGraphs}.
\end{df}

\begin{df}[$r$-Dominant outer automorphism]
Let $\vphi \in \out$ be fully irreducible and suppose $IW(\vphi)$ is $r$-dominant. Then we say that $\vphi$ is an \emph{(r-)dominant} outer automorphism.
\end{df}

\begin{rk}
Notice that if $\vphi$ is dominant then it is ageometric and by Theorem \ref{t:uniqueaxis} $\vphi$ is also a lone axis outer automorphism.
\end{rk}

\begin{df}[Stratum $\mS_r(\mG)$]\label{d:Stratum}
For a finite graph $\mG$, we define the \emph{stratum $\mS_r(\mG) \subset \periodiclines$ for $\mG$}:
$$\mS_r(\mG) :=\{\mL \mid \mL \text{ is an axis for a fully irreducible } \vphi \in Out(F_r) \text{ with } IW(\vphi) \cong \mG \}.$$
If $\mG$ is an $r$-dominant graph, we call $\mS_r(\mG)$ a \emph{dominant stratum}.
\end{df}

\begin{df}[$\mG$-basin $B\mS_r(\mG)$]\label{d:Basin}
For an $r$-dominant graph $\mG$, we define the \emph{$\mG$-basin}:
$$B\mS_r(\mG) := \{\mL \mid \mL \text{ is an axis for an ageometric fully irreducible } \vphi \in Out(F_r)$$ $$ \text{where } IW(\vphi) \text{ is the union of some subset of the components of } \mG \}.$$
\end{df}
Thus $B\mS_r(\mG)\subseteq \mS_r(\mG) \subset \periodiclines$. Notice that each element in $B\mS_r - \mS_r$ is not a lone axis automorphism since its index sum is strictly larger than $\frac{3}{2}-r$.

We give special names (and attention) to the following dominant strata.

\begin{df}[Principal strata $\mP_r$]\label{d:PrincipalStratum} Let $\pgraph$ denote the graph that is a disjoint union of $2r-3$ triangles. Notice that, in particular, $\pgraph$ has index sum $\frac{3}{2}-r$. We define the \emph{(rank-$r$) principal stratum} of $\periodiclines$ as $\mP_r = \mS_r(\pgraph)$.

In light of the above, we call a fully irreducible outer automorphism $\vphi \in Out(F_r)$ with $IW(\vphi) \cong \pgraph$ a \emph{principal} outer automorphism.

We define the \emph{(rank-$r$) principal stratum basin} in $\periodiclines$ as $B\mP_r :=B\mS_r(\pgraph)$. Outer automorphisms with axes in $B\mP_r$ will be called \emph{principal basin} outer automorphisms.
\end{df}
Note that
$$B\mP_r :=\{\mL \mid \mL \text{ is an axis for a fully irreducible } \vphi \in Out(F_r) \text{ having } IW(\vphi) \cong \Delta_{r'} \text{ with } r' \leq r \}.$$

\begin{rk}
We have noted that every dominant outer automorphism is a lone axis outer automorphism.
If an outer automorphism is principal then its axis intersects the interior of a maximum dimensional simplex in $\os$. This can be seen to follow from Proposition \ref{P:EveryVertexPrincipal}. Moreover, if $\vphi$ is dominant and not principal, then it will not pass through the interior of a maximum dimensional simplex, as one of its Whitehead graphs comes from a stable whitehead graph of a vertex with more than three stable directions.
\end{rk}

\begin{df}[Principal index list]\label{d:PrincipalIndexList}
Since the index list for a principal outer automorphism is comprised of terms $-\frac{1}{2}$ summing to $\frac{3}{2}-r$, we call this the \emph{principal} index list.
\end{df}

\section{Nielsen path prevention}{\label{s:NPPrevention}}

Definition \ref{d:LongTurns} is more or less in the spirit of \cite{cl15}.
We shall prove that under certain conditions a map $g$ is \emph{legalizing} (see Definition \ref{d:Legalizing}). Our goal will be to show that if a train track map $h$ factors through $g$ then it cannot admit a PNP.

\begin{df}[Long turns]\label{d:LongTurns} Suppose that we have a train track structure on $\Gamma$ induced by a trian track map $g$ on $\Gamma$ (see Definition \ref{d:GateStructures}). By a \emph{long turn} at a vertex $v$ we will mean a pair of legal paths $\{\alpha, \beta\}$ emanating from $v$. If $\{D(\alpha), D(\beta)\}$ is legal, then we call $\{\alpha, \beta\}$ \emph{legal}. If $\{D(\alpha), D(\beta)\}$ is illegal, then we call $\{\alpha, \beta\}$ \emph{illegal}.

If either $g(\alpha)$ is an initial subpath of $g(\beta)$ or vice versa, then we call $\{\alpha, \beta\}$ \emph{extendable}.
Those long turns that are not extendable can be characterized as either \emph{safe} or \emph{dangerous} depending on whether, respectively, $g_{\#}(\ol{\alpha} \beta)$ is a legal path or not (whether the cancellation of $g_{\#}(\alpha)$ and $g_{\#}(\beta)$ ends with a legal turn or an illegal turn).
\end{df}

The following is a relatively direct consequence of Lemma \ref{l:iNP}.

\begin{lem}{\label{l:longinps}}
Let $g \colon \Gamma\to\Gamma$ be an expanding irreducible train track map and $\rho$ an iNP for $g$. Then $\rho=\bar \rho_1\rho_2$, where $\{\rho_1,\rho_2\}$ is a dangerous long turn for each positive power $g^k$ of $g$. More generally, if $g \colon \Gamma\to\Gamma$ has a PNP, then $\Gamma$ contains dangerous long turns for each positive power $g^k$ of $g$. Thus, an expanding irreducible train track map with no dangerous long turns has no PNPs.
\end{lem}

\begin{df}[k-Protected path]\label{d:kProtected} Let $g \colon \Gamma \to \Gamma$ be an expanding irreducible train track map. Let $\gamma$ be a path in $\Gamma$ and $\alpha$ a subpath of $\gamma$ whose endpoints are at vertices. We say that $\alpha$ is \emph{k-protected} if
\begin{itemize}
\item $\gamma$ contains $\geq k$ edges to the right of $\alpha$ and $\geq k$ edges to the left of $\alpha$ and
\item the length-$k$ subpath of $\gamma$ directly to the right of $\alpha$ and the length-$k$ subpath of $\gamma$ directly to the left of $\alpha$ are each legal.
\end{itemize}
\end{df}

\begin{df}[Splitting]\label{d:Splitting} Let $g \colon \Gamma \to \Gamma$ be an expanding irreducible train track map. Let $\gamma$ be a path in $\Gamma$. We say that $\gamma=\dots\gamma_{l-1}\gamma_l\dots$ is a \emph{k-splitting} if $g_{\#}^k(\gamma)= \dots g_{\#}^k(\gamma_{l-1})g_{\#}^k(\gamma_l)\dots$ is a decomposition into subpaths. $\gamma$ is a \emph{splitting} if it is a $k$-splitting for all $k>0$.
\end{df}

The following is a special case of the definition of $P_r$ on pg. 558 of \cite{bfh00}.

\begin{df}[$Pth_g$]\label{d:PreNielsenPaths} Let $g \colon \Gamma \to \Gamma$ be an expanding irreducible train track map. We let $Pth_g$ denote the paths in $\gamma$ so that:
\begin{enumerate}
\item Each $g_{\#}^k(\gamma)$ contains exactly one illegal turn.
\item The number of edges in $g_{\#}^k(\gamma)$ is bounded independently of $k$.
\end{enumerate}
\end{df}

The following is \cite[Lemma 4.2.5]{bfh00}.

\begin{lem}[\cite{bfh00}]{\label{l:PFinite}}
Let $g \colon \Gamma \to \Gamma$ be an expanding irreducible train track map. Then $Pth_g$ is finite.
\end{lem}

The next lemma states that there is a uniform $k$ so that for each long turn $\{\al, \beta\}$ the iterate $g^k(\bar \al \beta)$ splits into at most three well understood parts.

\begin{lem}{\label{l:LongTurnSplittingPower}}
Suppose that $g$ is an expanding irreducible train track map representing a fully irreducible outer automorphism $\vphi \in Out(F_r)$. Then there exists some power $g^k$ of $g$ so that for each long turn $\{\alpha,\beta\}$ for $g$ we have that $g^k_{\#}(\bar\alpha\beta)$ splits with respect to $g$ into paths that are legal paths except for at most a single iNP.
\end{lem}

\begin{proof}
First notice that, since $\vphi$ is fully irreducible, $g \colon \Gamma \to \Gamma$ has only a single EG stratum. Thus, by \cite[Lemma 4.2.2]{bfh00}, there exists a constant $K$ so that, if $\tau$ is a path in $\Gamma$ and if $\sigma$ is a $K$-protected subpath of $\tau$, then $\tau$ can be split at the endpoints of $\sigma$.

There are only finitely many paths of length $\leq 4K$. Let $\mS$ denote the set of all paths of length $\leq 4K$ with only a single illegal turn. Let $k \in \Z$ be the power so that, if $\rho \in \mS$, then either $g_{\#}^k(\rho)$ is legal or for no $n \in \Z$ is $g_{\#}^{nk}(\rho)$ legal. By \cite[Lemma 4.2.6]{bfh00} we can then replace $k$ with a higher power, if necessary, so that for each $\rho \in \mS$, we have that $g_{\#}^k(\rho)$ splits into subpaths that are either legal or an element of $Pth_g$ (a uniform power is possible since $\mS$ is finite). The subpaths that are elements of $Pth_g$ are permuted. Thus, by replacing $k$ with a higher power yet, we can assume that they are iNPs (hence have only one illegal turn).

Let $\{\alpha, \beta\}$ be a long turn. Then using trivial paths as $K$-protected subpaths, $\overline{\alpha}\beta$ can be split into legal paths and a path $\rho$ of length $\leq 4K$ containing the single illegal turn. Since $\rho \in \mS$, we can use the power $k$ of the previous paragraph and obtain that $g^k_{\#}(\bar\alpha\beta)$ splits into subpaths that are either legal or iNPs. But $\overline{\alpha}\beta$ had only one illegal turn, and the number of illegal turns cannot increase under $g^k_{\#}$. So there can only be one iNP in the splitting.
\qedhere
\end{proof}

\begin{lem}{\label{l:DangerousLongTurnsDisappearUnderHighIteration}}
Suppose that $g$ is an expanding irreducible train track map with no PNPs. Then there exists some power $g^k$ of $g$ with no dangerous long turns.
\end{lem}

\begin{proof}
Let $k$ be as in Lemma \ref{l:LongTurnSplittingPower} and suppose, for the sake of contradiction, that $g^k$ had a dangerous long turn $\tau =\{\alpha, \beta\}$. By Lemma \ref{l:LongTurnSplittingPower}, $g_{\#}^k(\bar\alpha\beta)$ splits into legal paths and iNPs. Since $g^k$ admits no iNPs, $g_{\#}^k(\bar\alpha\beta)$ is legal, contradicting that $\tau$ is dangerous.
\qedhere
\end{proof}

\begin{df}[Legalizing train track maps]{\label{d:Legalizing}} We call a train track map $g \colon \Gamma \to \Gamma$ \emph{legalizing} if it has no dangerous long turns.
\end{df}

\begin{prop}{\label{p:LegalizingMaps}}
Suppose that $g$ is a PNP-free expanding irreducible train track map. Then there exists some $p>0$ so that $g^p$ is legalizing.
\end{prop}

\begin{proof}
This follows from Lemmas \ref{l:longinps} and \ref{l:DangerousLongTurnsDisappearUnderHighIteration}.
\qedhere
\end{proof}

\begin{prop}{\label{p:ConvenientReps}}
Suppose $\vphi$ is a lone axis fully irreducible outer automorphism. Then there is a fully stable transparent legalizing train track representative $g \colon \Gamma \to \Gamma$ of a power $\vphi^R$ of $\vphi$ so that all vertices of $\Gamma$ are principal and fixed, and all but one direction is fixed.
\end{prop}

\begin{proof}
This follows from Remark \ref{noPNP}, Proposition \ref{P:EveryVertexPrincipal}, and Proposition \ref{p:LegalizingMaps}.
\qedhere
\end{proof}

\begin{df}[Convenient train track maps]{\label{d:Convenient}} For a lone axis fully irreducible outer automorphism $\vphi$ we call a train track representative of a power $\vphi^R$ of $\vphi$ satisfying the properties of Proposition \ref{p:ConvenientReps} \emph{convenient}.
\end{df}

\section{Proof of the main result}{\label{s:Proof}}

\begin{lem}{\label{l:SkeletaSandwich}}
Let $r\ge 3$ be an integer, so that $n=3r-4$ is the dimension of $\os$.
For each lone axis fully irreducible $\vphi \in Out(F_r)$ there exists an integer $k \leq n$ so that $A_{\vphi} \subset \os^{(k)}\setminus \os^{(k-2)}$. In particular, all folds in $A_{\vphi}$ are proper full folds.
\end{lem}

\begin{proof}
Let $g \colon \Gamma \to \Gamma$ be a convenient train track map  representing $\vphi$, guaranteed by Proposition \ref{p:ConvenientReps}.
Let $k=|E(\Gamma)|$, i.e. $1+$ the dimension of the open simplex containing $\Gamma$. The fold line $A_\vphi$ is a periodic fold line for a Stallings fold decomposition of $g$.
It suffices to show that each fold of $A_{\vphi}$ is a proper full fold.
But, if one of the folds were full, then some vertex of $\Gamma$ would not be $g$-periodic, hence not principal. This contradicts that $g$ is convenient.
\qedhere
\end{proof}

If $\mathfrak{g}$ is a Stallings fold decomposition of $g$ we denote by $\mathfrak{g}^p$ the Stallings fold sequence obtained by juxtaposing $p$ copies of $\mathfrak{g}$. Note that $\mathfrak{g}^p$ is a decomposition of $g^p$.
In the next lemma we need not assume the outer automorphism represented by $h$ is fully irreducible.

\begin{lem}{\label{l:Power}}
Let $r\ge 3$ be an integer, so that $n=3r-4$ is the dimension of $\os$, and let $2\le k \leq n$.
Suppose that $\mL \subset \os^{(k)}\setminus \os^{(k-2)}$ is the periodic fold line for a Stallings fold decomposition $\mathfrak{g}$ of a train track map $g$. Then there exist constants $R,\veps>0$ so that:
For each fold line $\mL' \in B^k(\mL,R,\veps)$ that is the periodic fold line corresponding to a Stallings fold decomposition $\mathfrak{h}$ of some train track map $h$, there exists a power $p$ so that $\mathfrak{h}^p$ contains $\mathfrak{g}$.

In particular, $h$ and $g$ are self-maps of the same topological graph $\Gamma$.
\end{lem}

\begin{proof}
Let $R$ be three times the length of a $\mathfrak{g}$-segment of $\mL$. Since $\mL$ is periodic and contained in $\os^{(k)}\setminus \os^{(k-2)}$, there exists some $\veps_0>0$ such that $N(\mL,\veps_0) \cap \os^{(k)} \subset \os^{(k)}\setminus \os^{(k-2)}$. Therefore, for any $0<\veps \le \veps_0$, any geodesic segment $\gamma$ of length $R$ contained in $N(\mL,\veps)$ passes through the same sequence of simplices as a subsegment of $\mL$ of length $\geq R-2\veps$, and hence shares a fold sequence with this subsegment of $\mL$.
Choose $\veps=min\{\veps_0,\frac{R}{6}\}$. Then any subsegment of $\mL$ of length $\geq R-2\veps \geq \frac{2}{3}R$ contains twice the length of a $\mathfrak{g}$-segment of $\mL$ so must contain a $\mathfrak{g}$-segment of $\mL$. Hence, any periodic fold line $\mL' \in B^k(\mL,R,\veps)$ will in fact contain the full fold sequences $\mathfrak{g}$.
We can now take the power $p$ of $\mathfrak{h}$ high enough so that $\mathfrak{h}^p$ contains any length-$R$ subsegment of $\mL'$ and the conclusion of the theorem will hold.
\qedhere
\end{proof}

\begin{lem}{\label{l:Neighborhood}}
Let $r\ge 3$ be an integer, so that $n=3r-4$ is the dimension of $\os$, and let $2\le k \leq n$.
Suppose $\vphi$ is a dominant lone axis fully irreducible with axis $A_{\vphi} \subset \os^{(k)} \setminus \os^{(k-2)}$.
Then there exist constants $R,\veps>0$ and a convenient train track representative $g\from \Gamma \to \Gamma$ of a power $\vphi^p$ of $\vphi$ so that for each periodic fold line $\mL\in B^k(A_\vphi,R,\veps)$ there exist
and a self-map $h$ on $\Gamma$ with a Stallings fold decomposition yielding $\mL$ and such that:
\begin{enumerate}[label=(\alph*)]
\item $h \from \Gamma \to \Gamma$ is a train track map.
\item $h$ does not admit a PNP.
\item The transition matrix for $h$ is Perron-Frobenius.
\item $\displaystyle \bigcup_{v \in V(\Gamma)} LW(g,v) =  \bigcup_{v \in V(\Gamma)} LW(h,v)$.
\item If the vertex $w$ of $\Gamma$ is $h$-periodic then $SW(h,w) = SW(g,w)$.
\item If $\mL$ contains no proper full fold, then all vertices of $\Gamma$ are principal with respect to both $g$ and $h$ and $\displaystyle \bigcup_{v \in V(\Gamma)} SW(g,v) = \bigcup_{v \in V(\Gamma)} SW(h,v)$.
\end{enumerate}
\end{lem}

\begin{proof}
Since $\vphi$ is a lone axis fully irreducible, by Proposition \ref{P:EveryVertexPrincipal}, there exists a rotationless power $\vphi^p$ of $\vphi$ with a convenient train track representative $g \colon \Gamma \to \Gamma$. We call the nonfixed direction $d$. Since $\vphi$, hence $\vphi^p$, is a lone axis fully irreducible, $A_{\vphi}$ is the unique periodic fold line for $g$ and is formed by iterating the fold sequence $\mathfrak{g}$ for $g$. Replace $g$ with a power so that each turn in $LW(g)$ is taken by $g(e)$ for each edge $e$.

Notice that $A_{\vphi}$ is also the periodic fold line for $g^3$ and that $\mathfrak{g}^3$ is the fold sequence for $g^3$. Applying Lemma \ref{l:Power}, there exist $R,\veps>0$ so that for any periodic fold line $\mL \in B^k(A_\vphi,R,\veps)$ corresponding to a train track map $h'$ and fold sequence $\mathfrak{h}'$ of $h'$, there exists a power $p$ such that $(\mathfrak{h}')^{p}$ contains $\mathfrak{g}^3$. Thus, replacing $h$ with this power and possibly applying a cyclic permutation, $h'$ factors as $h' = f \circ g^3$, see (\ref{e:L'}).

The trickiest aspect of this proof, and the reason to use $g^3$ instead of $g$, is to prove item (b). We will show all items for the cyclic permutation $h = g \circ f \circ g^2$ of $h'$.

\begin{equation}{\label{e:L'}}
\xymatrix{\Gamma \ar[r]_{g} \ar@/_2pc/[rrrr]_{h} & \Gamma \ar[r]_{g} & \Gamma \ar[r]_{f}  & \Gamma   \ar[r]_{g}  \ar@/^2pc/[rrrr]^{h'}&
 \Gamma  \ar[r]_{g}  & \Gamma  \ar[r]_{g}  & \Gamma  \ar[r]_{f}  & \Gamma  \\}
\end{equation}

We first show (a). Suppose $h$ is not a train track map, i.e. $h^{p}(e)$ contains a backtracking segment for some $e\in E(\Gamma)$ and power $p$.
We parametrize $\mL \from \RR \to \os$ so that the graphs appearing in (\ref{e:L'}) are $\mL(0), \mL(1), \mL(2), \dots$ respectively.
Let $\gamma \in F_r$ be the witness guaranteed by Lemma \ref{l:WitnessLoops}, i.e. $\gamma_t \in \mL(t)$ is legal for all $t \geq 0$. Note that $\gamma_4 = g(\gamma_3)$, and since $g$ maps each edge onto the entire graph, $\gamma_4$ contains $e$. Thus, $e$ cannot be $h$-legal, a contradiction.

We now show (b). Recall that $g$ is convenient, hence if $\{\al,\beta\}$ is a long turn then either $g(\al)$ is an initial subsegment of $g(\beta)$ or $g_\#(\bar \al \beta) = \bar\al' \beta'$ is legal, where $\al', \beta'$ are nontrivial terminal subsegments of $\al,\beta$. In the second case all turns of $g(\bar \al' \beta')$ are $g$-taken (since $\{Dg\al',Dg\beta'\}$ cannot contain $d$ so it, too, is $g$-taken). Notice that each $g$-taken turn is $h^{p} \circ g \circ f$-legal since for any witness loop $\gamma$, $g^2(\gamma)$ maps over all $g$-taken turns and  $h^{p} \circ g \circ f(g^2(\gamma))$ is legal. Concluding, we get that the path $g(\bar \al' \beta')$ is legal with respect to $h^{p} \circ g \circ f$ for each $p$.
Now if $\rho = \bar \al \beta$ is an iPNP, then for some $p$, $h^{p}_\#(\rho) = \rho$, which is illegal. But $h^{p}(\bar \al \beta) =  h^{p-1} \circ g \circ f \circ g ( g_\#(\bar\al\beta)) =  h^{p-1} \circ g \circ f( g(\bar\al'\beta'))$, which is legal. We get a contradiction to the fact that $\rho$ is an iPNP.

To prove (c) recall that each edge of $\Gamma$ maps onto $\Gamma$ under the map $g$. Since $h = g \circ (f \circ g^2)$ is a train track map, the same is true for $h$. Thus the transition matrix of $h$ is PF.

To prove (d), recall that $g(e)$ contains all turns in each local Whitehead graph. Note also that for any witness loop $\gamma$ for $\mL$, $h(\gamma)$ contains all $g$-taken turns. Thus $LW(h,w) \supset LW(g,w)$. Let $d$ be the unique $g$-nonperiodic direction, and let $v$ be its initial vertex. Note that $\cup_{w \in \Gamma} LW(g,w)$ contains all turns not involving $d$. Thus, if $\cup LW(h,w) \setminus \cup LW(g,w) \neq \emptyset$ then $\{d,d'\} \in LW(h,v) \setminus LW(g,v)$ for some $d' \neq d$. Therefore, there exists an edge $e$ and a natural $p$ so that $h^p(e)$ crosses $\tau =\{d,d'\}$. Denoting $\al = f \circ g^2 \circ h^{p-1}(e)$ (which is an immersed $g$-legal path) we have $h^p(e) = g(\al)$ contains $\tau$. But $\tau$ is not $g$-taken and not in $Im(Tg)$, since $d \notin Im(Dg)$, a contradiction. So $LW(h,w) = LW(g,w)$ for each $w \in \Gamma$.

To prove (e) we again denote by $v$ the initial vertex of the direction $d$ that is $g$-nonperiodic. First let $w \neq v$ be $h$-periodic. Since all turns at $w$ are $h$-taken, $Dh$ is injective on the directions at $w$, hence all directions at $w$ are periodic and $SW(h,w) = LW(h,w) = LW(g,w) = SW(g,w)$.
For $v$, since $LW(v,g) = LW(v,h)$, we have that $LW(v,h)$ contains a complete graph $C$ on $deg(v)-1$ vertices (all directions except $d$). $Dh$ is injective on $C$, since otherwise a taken turn would be illegal. If $v$ is not $h$-periodic then there is nothing to prove, so we assume that $v$ is $h$-periodic. Let $h^p$ be a rotationless power of $h$. Then $Dh^p \from LW(h,v) \to SW(h,v)$ sends $C$ to an isomorphic graph. Moreover, we know that $deg(v) = V(C)+1$ and that $d$ is in an $h$-illegal turn, so $SW(h,v)$ cannot contain more than $V(C)$ vertices. Hence, $SW(h,v) \cong C \cong SW(g,v)$.

To prove (f) note that the containment in (e) is proper if and only if not all vertices are principal. Since the local Whitehead graphs contain all edges without $d$, this happens only when for some $w \neq u \in V(\Gamma)$ we have $h^p(w) = h^p(u)$. When this happens, some fold in the fold sequence does not restrict to an injective map on the vertices. This implies that the fold is full. Therefore, if no fold in the fold sequence of $\mL$ is full, then all vertices of $\Gamma$ are $h$-principal, and the stable Whitehead graphs of $h$ and $g$ are identical.
\qedhere
\end{proof}

The basin of any dominant stratum has the following ``rigidity'' properties:

\begin{mainthmA}{\label{T:MainTheorem1}}
Let $r\ge 3$ and let $\mathcal G$ be an $r$-dominant graph. Let $\mL \in \mathcal S^r(\mathcal G)$. Then there exist $0\le k\le 3r-4$ and a neighborhood $U\subseteq \periodiclines$ of $\mL$ with the following properties:

\begin{itemize}
\item[(a)]For each $\mL'\in U$ with $\mL'\subseteq \os^{(k)}$, we have $\mL'\in B\mathcal S^r(\mathcal G)$.
\item[(b)] For each $\mL'\in U$ with $\mL'\subseteq \os^{(k)}$ and with $\mL'$ containing no full folds, we have $\mL'\in \mathcal S^r(\mathcal G)$.
\end{itemize}
\end{mainthmA}

\begin{proof}
Using Lemma \ref{l:Neighborhood}, we have that $\mL$ is a periodic fold line for a Stallings decomposition of the train track map $h$ with no PNPs and with connected local Whitehead graphs. By the FIC (Proposition~\ref{prop:FIC}) we get that the outer automorphisms $\vphi'$ represented by $h$ is ageometric fully irreducible.
Let $W \subset V(\Gamma)$ be the set of $h$-principal vertices of $\Gamma$. Then by Lemma \ref{l:Neighborhood}(e),
\[ IW(\vphi') = \bigcup_{v \in W} SW(h,v) \subset   \bigcup_{v \in V(\Gamma)} SW(g,v) = IW(\vphi) \]
and the unions are disjoint. Thus, $IW(\vphi')$ is a union of components of $IW(\vphi)$.

We prove (b). By Proposition \ref{l:Neighborhood}(f), all vertices are $h$-principal, hence $IW(\vphi') = \bigcup_{v \in V(\Gamma)} SW(h,v) = \bigcup_{v \in V(\Gamma)} SW(g,v) = IW(\vphi)$.
\qedhere
\end{proof}

\begin{cor} Suppose $\mL \in \mP_r$. Then there exists an open neighborhood $U$ of $\mL$ so that:
\begin{itemize}
\item[(a)] $U \cap \periodiclines \subset B\mP_r$ and
\item[(b)] each periodic fold line in $U$ containing no proper folds is contained in $\mP_r$.
\end{itemize}
\end{cor}

\section{Examples}{\label{s:Examples}}

\begin{ex}[Principal outer automorphisms exist]\label{example1}
We claim that, for each rank $r \geq 3$, the examples constructed in \cite{cl15} to have the principal index list are in fact principal outer automorphisms. For each $r \geq 3$, we denote this outer automorphism in $Out(F_r)$ by $\vphi_r$. By \cite[Theorem 6.2]{cl15} we know that each $\vphi_r$ is an ageometric fully irreducible outer automorphism. To show that $\vphi_r$ is principal we must prove that $IW(\vphi_r)=\pgraph$.

The proof of \cite[Proposition 4.3]{cl15} indicates that the stable Whitehead graph at each vertex is a complete graph. Since there are no periodic Nielsen paths (again by \cite[Theorem 6.2]{cl15}), this indicates that each component of the ideal Whitehead graph is a complete graph. Since the map is constructed to have the principal index list, this implies that the ideal Whitehead graph is in fact $\pgraph$, as desired.
\end{ex}

\begin{ex}\label{example2}
We give an example of an ageometric fully irreducible outer automorphism $\vphi'$ that is a composition $\psi \circ \vphi$ where:
\begin{itemize}
\item $\vphi$ is a principal outer automorphism and
\item $\vphi'$ is not principal, but is only a principal basin outer automorphism.
\end{itemize}

This example reveals the necessity on our restricting in Theorem A to fold lines consisting of proper full folds and, in so doing, indicates that, unlike in the \teich space setting, being in the principal stratum is not quite an ``open'' condition.

The outer automorphism $\vphi$ is $\vphi_3$ of Example \ref{example1} (from \cite{cl15}). Let $f\from \Gamma_1 \to \Gamma_1$ be the  train track representative used by Coulbois-Lustig to define it (see Figure \ref{f:NotOpen} below). We use the notation $e_1, e_2, e_3, e_4, e_5$ to respectively  denote the edges $\bar{c_3}, \bar{c_2}, \bar{c_1}, \bar{a_1}, \bar{b_1}$ in \cite[\S 3]{cl15}(see ``maximal odd'' case on pages 1117-1118). We replace $f$ with a high enough power to be both transparent and legalizing in the sense of Definition \ref{d:Legalizing} (see Proposition \ref{p:LegalizingMaps}). The only illegal turn of $f$ is $\{\bar e_3, \bar e_4 \}$ and all other gates are singleton directions.

We shall define a map $k\from \Gamma_1 \to \Gamma_1$ as a composition of folds, namely $g_1, g_2, g_3$ and a homeomorphism $h$, and we define $f' := k \circ f$. The map $f'$ will be shown to be a train track map representing an ageoemtric fully irreducible $\vphi' \in \out$.
\begin{equation}{\label{e:f'}}
\xymatrix{\Gamma_1 \ar[r]_{f} \ar@/_3pc/[rrrrr]_{f'} & \Gamma_1 \ar@/^2pc/[rrrr]^{k} \ar[r]_{g_1} & \Gamma_2 \ar[r]_{g_2} & \Gamma_3   \ar[r]_{g_3}  & \Gamma_4  \ar[r]_{h}   & \Gamma_1
\\}\end{equation}
The maps $g_1, g_2, g_3, h$ are described in Figure \ref{f:NotOpen}.
Composing the maps in the diagram yields:
\[ k(e_1) = \bar{e_1} \bar{e_4}, \quad
k(e_2) = e_1, \quad
k(e_3) = e_2 \bar{e_5}, \quad
k(e_4) = e_2 \bar{e_5}, \quad
k(e_5) = \bar{e_3} \bar{e_5}
\]
\begin{figure}[hb]\label{f:NotOpen}
\includegraphics[width=1.25in]{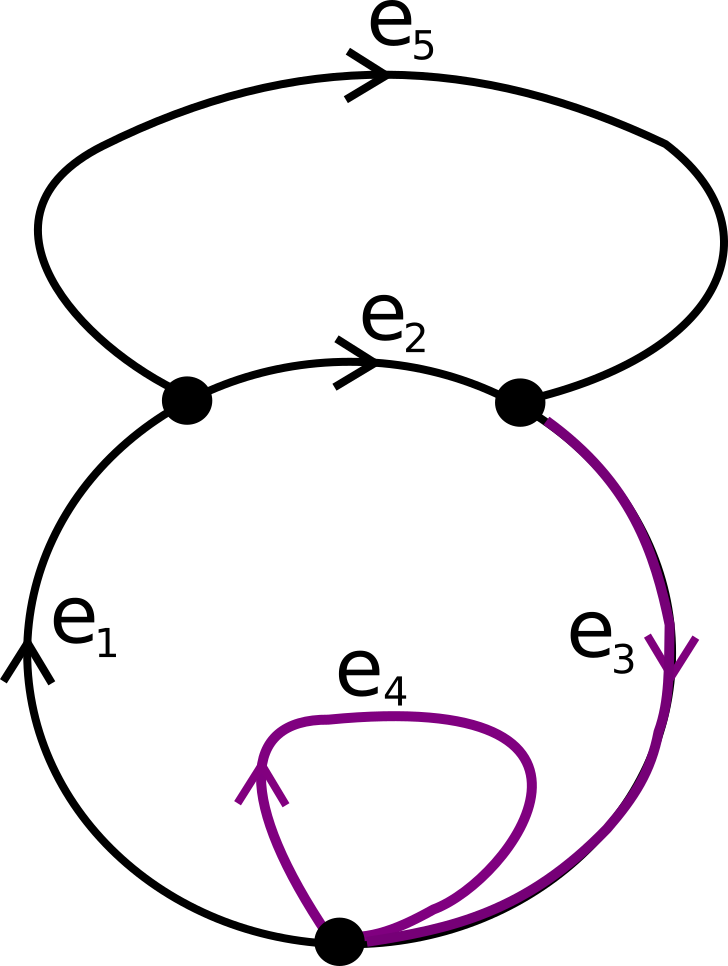}
\includegraphics[width=.4in]{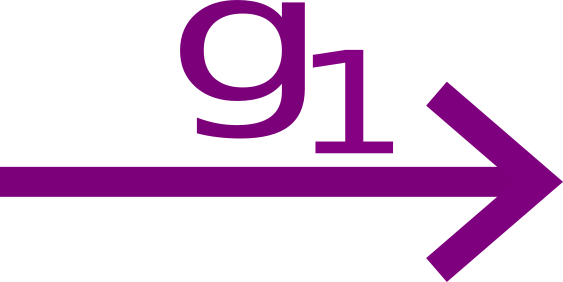}
\includegraphics[width=1.25in]{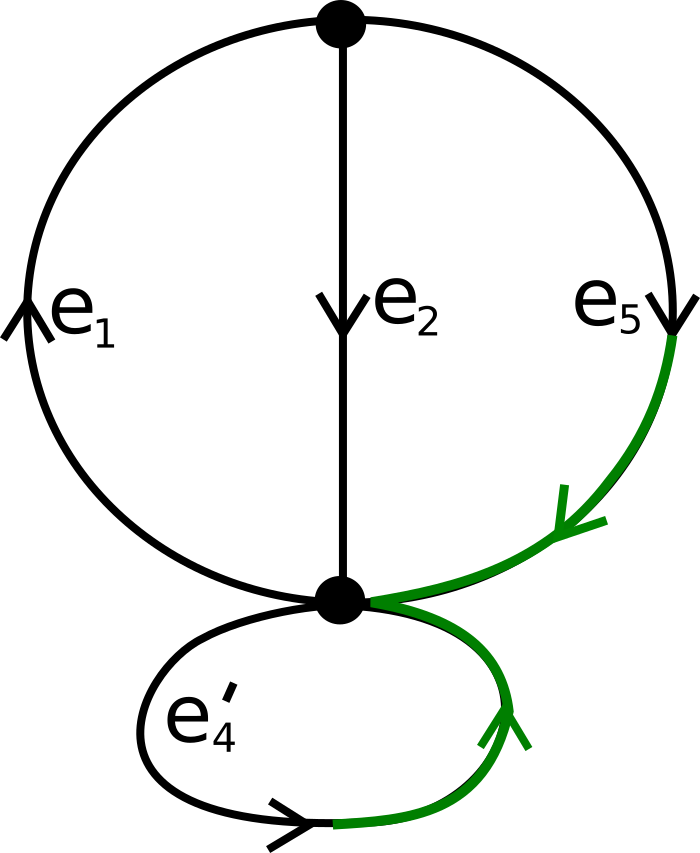}
\includegraphics[width=.4in]{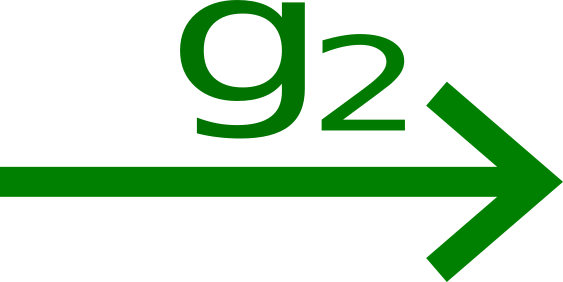}
\includegraphics[width=1.25in]{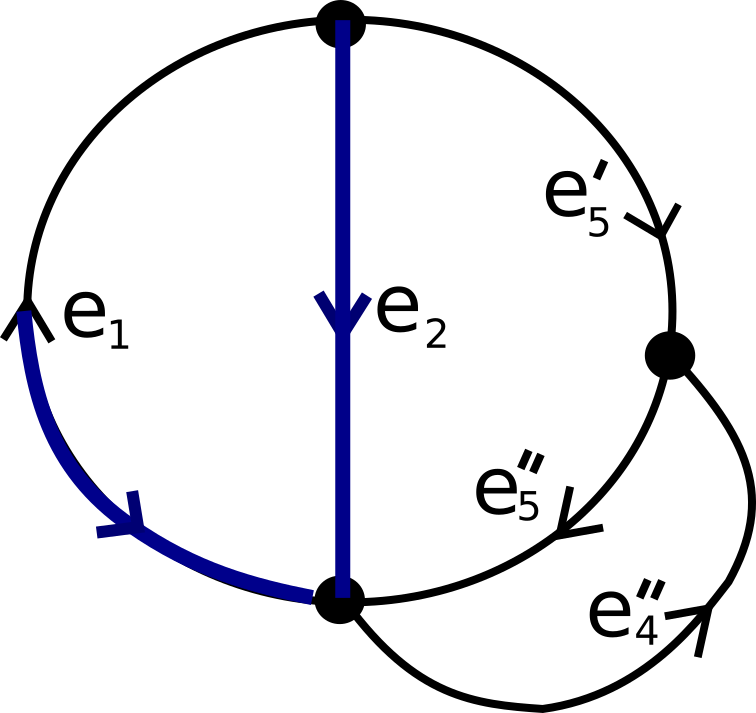}
\includegraphics[width=.4in]{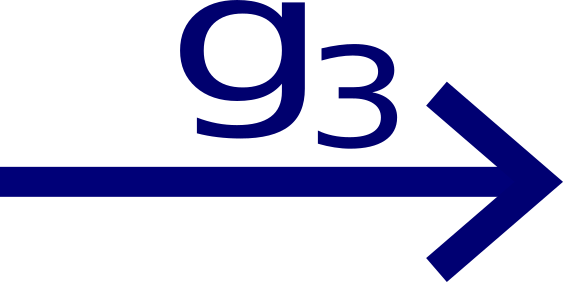}
 \includegraphics[width=1.25in]{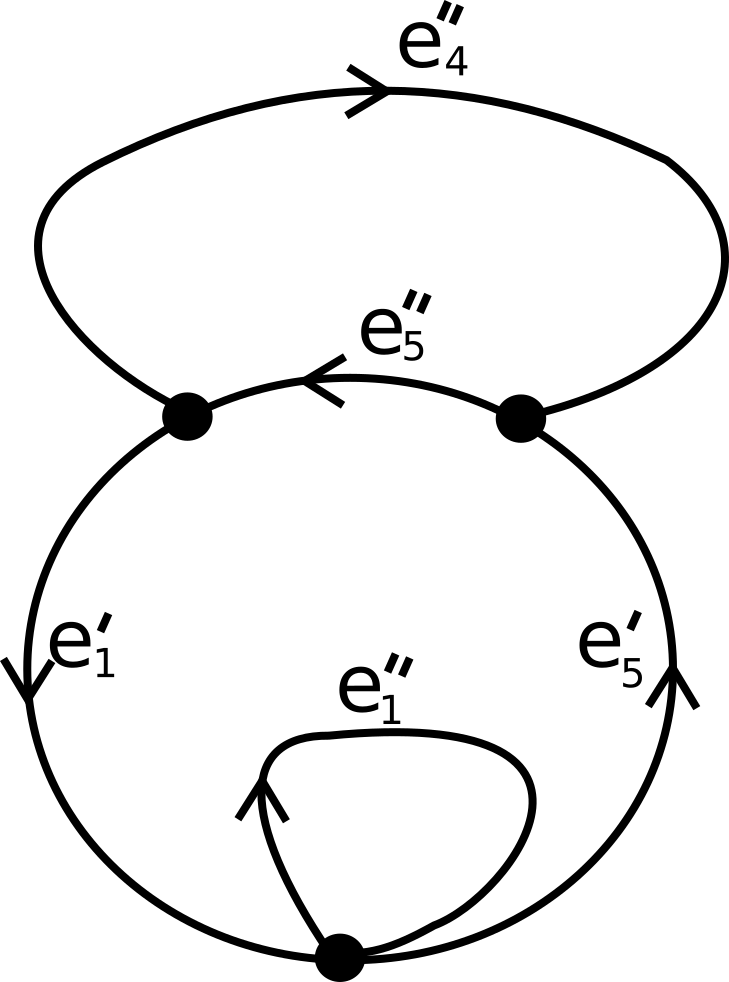}
\includegraphics[width=.4in]{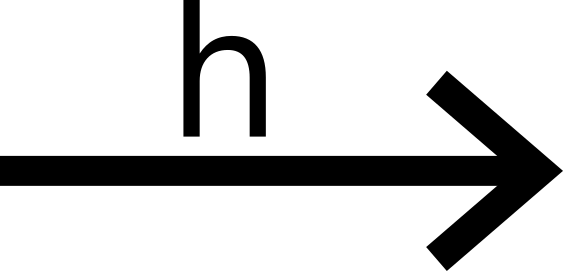}
 \includegraphics[width=1.25in]{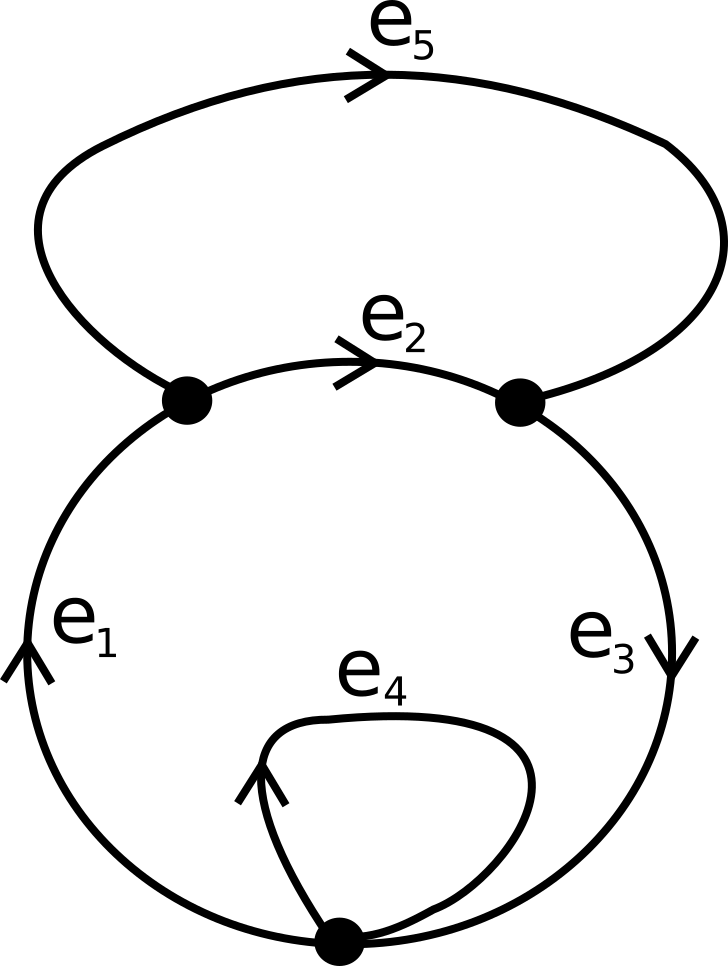}
 \caption{From left to right the graphs are $\Gamma_1$, $\Gamma_2$, $\Gamma_3$, $\Gamma_4$, $\Gamma_1$. Colors indicate the (partial) edges folded by the subsequent maps. $g_1$ is the full fold of $\ol{e_3}$ and $\ol{e_4}$ (and $e_4'$ is the edge formed by the identification of $\ol{e_3}$ and $\ol{e_4}$). $g_2$ is a partial fold of $\ol{e_4'}$ and $\ol{e_5}$ (and $e_5''$ is the edge formed from the identification of the initial portions of $\ol{e_4'}$ and $\ol{e_5}$, and $e_5'$ is the portion of $e_5$ not folded, and $e_4''$ is the portion of $e_4$ not folded). $g_3$ is a proper full fold of $e_1$ over $\ol{e_2}$ (and $e_1''$ is the portion of $e_1$ remaining after the fold), i.e. $g_3(e_1) = e_1' e_1''$. $h$ is a homeomorphism sending $e_1''$ to $\ol{e_4}$, and $e_1'$ to $\ol{e_1}$, and $e_4''$ to $e_2$, and $e_5''$ to $\ol{e_5}$, and $e_5'$ to $\ol{e_3}$. The bottom vertex in $\Gamma_1$ is $v_1$, the upper left vertex is $v_2$, and the upper right vertex is $v_3$.}
\end{figure}

Notice the following facts:
\begin{enumerate}
\item The only $k$-prenull turn is $\{\bar{e_3}, \bar{e_4}\}$.
\item $\{\bar{e_3}, \bar{e_4}\}$ is the only $f$-illegal turn, and it is not in $Im(Tf)$.
\item $Im(Dk)$ does not contain $\bar{e_4}$.
\item $\{\bar{e_3}, \bar{e_4}\}$ is not a $k$-taken turn. The $k$-taken turns are: $\{\bar{e_2},\bar{e_5}\}, \{e_3, \bar{e_5}\}, \{e_1, \bar{e_4}\}$.
\item $k(v_1)= k(v_3) = v_2$ and $k(v_2) = v_1$ (see the bottom of the label of Figure \ref{f:NotOpen}).
\item The $Tk$ images of all of the $f$-legal turns at $v_1$ and $v_2$ are:
$\{e_2, e_5\}, \{\bar{e_1}, e_5\},
 \{\bar{e_1}, e_2\}$ at the vertex $v_2$ and
$\{e_1, \bar{e_3}\}, \{e_1, e_4\}, \{e_4, \bar{e_3}\}$
 at the vertex $v_1$.
\item $f$ takes all turns not involving $\bar e_4$.
\end{enumerate}
\end{ex}

\begin{lem}
$f'=k \circ f$ is an irreducible train track map.
\end{lem}

\begin{proof}
We first show that $f'$ is a train track map. Suppose not - then for some $p \in \NN$ and edge $e$, $(f')^p(e)$ would contain backtracking. Note that then $f \circ (f')^{p}(e)$ would contain backtracking. We show by induction on $p$ that $f \circ (f')^{p}(e)$ does not contain backtracking. Since $f$ is a train track map, $f(e)$ can contain no backtracking. Inductively assume $\beta = f \circ (f')^{p-1}(e)$ has no backtracking. All of the turns in $\beta$ are either $f$-taken or in $Im(Tf)$, so also $f$-taken. No $f$-taken turn is $k$-prenull, so $k(\beta)$ has no backtracking. All turns in $k(\beta)$ are either $k$-taken or in $Im(Tk)$ so by properties (3) and (4) above, $k(\beta)$ does not contain $\{\bar e_3, \bar e_4\}$. So by (2) it is $f$-legal. This completes the induction step.
The fact that $f'$ is irreducible follows from the fact that $f(e)$ contains all edges and $k$ is onto.
\end{proof}

\begin{lem}{\label{l:iNPsDisappearUnderHighIteration}}
$f'=k \circ f$ has no PNPs.
\end{lem}

\begin{proof}
Recall that $f$ is legalizing and transparent, which implies that if $\al$ and $\beta$ are legal paths initiating at the same vertex then, without loss of generality, either $f(\al)$ is an initial subpath of $f(\beta)$ (then $\{\al,\beta\}$ is an $f$-extendable long turn) or $f_\#(\bar\al \beta) = \bar \al' \beta'$ is $f$-legal, where $\al', \beta'$ are terminal subsegments of $\al,\beta$ respectively.

Now suppose $\rho = \bar\al\beta$ is an iPNP for $f'$. Since there exists a $p>0$ so that $(f')^p_\#(\rho) = \rho$, we would then have that $\rho$ is not $f'$-extendable, which would imply that it is not $f$-extendable. By the first paragraph we have $f_\#(\rho) = \bar\al'\beta'$, where $\al', \beta'$ are $f$-taken paths and the turn $\{D\al', D\beta'\}$ is $f$-legal, so it is not equal to $\{\bar e_3, \bar e_4\}$, the only $k$-prenull turn. Thus $k (\bar \al' \beta')$ contains no backtracking. Therefore, the turns in $f'(\rho) = k (\bar \al' \beta')$ are in the image of $Tk$ or are $k$-taken turns. By properties (3) and (4) above, $f'(\rho)$ does not contain $\{\bar e_3, \bar e_4\}$, so is $g$-legal. Continuing this way, applying $k$ and $g$ alternately to $k(\al'\beta')$, we see that the turn $\{\bar e_3, \bar e_4\}$ never appears, so $f'(\rho)$ is $f'$ legal, contradicting the assumption that $\rho$ is a PNP.
\qedhere
\end{proof}

\begin{lem}{\label{l:PathologyFI}}
Let $f'=k \circ f$ represents an ageometric fully irreducible outer automorphism.
\end{lem}

\begin{proof}
By the FIC, it suffices to show that each local Whitehead graph is connected.

By item (6), we have that all turns at $v_2$ are $f'$-taken and all turns at $v_1$ not involving $d=\bar{e_4}$ are $f'$-taken. It follows that $LW(f',v_2)$ is a triangle and $LW(f',v_1)$ contains a triangle. Since $f'$ is an irreducible train track map and for no such map does $LW(v_1,f')$ have an isolated vertex, we have that $\bar e_4$ in $LW(v_1,f')$ is also connected via an edge to another vertex, hence $LW(v_1,f')$ is connected. By (4), $LW(f',v_3)$ is also connected.
\qedhere
\end{proof}

\begin{lem}
Let $f'=k \circ f$ and suppose that $f'$ represents the automorphism $\vphi' \in Out(F_3)$. Then the ideal Whitehead graph $IW(\vphi')$ is a union of two triangles.
\end{lem}

\begin{proof}
By the proof of Lemma \ref{l:PathologyFI}, $LW(f',v_2)$ is a triangle and $LW(f',v_1)$ contains a triangle. Now $v_1$ and $v_2$ are permuted by $f$ and $k$, hence are permuted by $f'$. Thus they are fixed by a rotationless power $f'^p$ of $f'$. Moreover, $Df'^p$ cannot identify any of the directions at $v_2$, since that would collapse a taken turn. Therefore, the triangle in $LW(f',v_2)$ is taken by $f'^p$ to a triangle in $SW(f'^p, v_2)$.
Similarly, no two of the directions $\{d_1, d_2, d_3 \}$ at $v_1$, distinct from $d$, can be identified by $Df'^p$. Thus their images span a triangle in $SW(f'^p,v_1)$.
Thus $SW(f',v_1)$ contains a triangle. But since $d$ forms an illegal turn with some other direction, there are only 3 gates, and the graph is in fact a triangle.
Lastly note that $k(v_3)= v_2$, so $v_3$ is not a principal vertex (this is in fact the key for dropping from a union of 3 triangles in $IW(f)$ to two triangles in $IW(f')$).
The ideal Whitehead graph is the union of the stable Whitehead graphs of principal vertices glued along PNPs but, since $f'$ admits none, $IW(\vphi')$ is a union of two disjoint triangles.
\end{proof}

We can now prove one of the main results stated in the introduction:

\begin{mainthmB}
There exists a principal fully irreducible outer automorphism $\vphi \in Out(F_3)$ with a train track representative $f \colon \Gamma \to\Gamma$ with a Stallings fold decomposition $\mathfrak f$ such that, for each $n\ge 1$, there exists a nonprincipal fully irreducible outer automorphism  $\psi_n \in Out(F_3)$ with a train track representative $g_n \colon \Gamma \to\Gamma$ with a Stallings fold decomposition $\mathfrak g_n$ such that $\mathfrak g_n$ starts with $\mathfrak f^{n}$.
\end{mainthmB}

\begin{proof} The lemmas and proofs above apply verbatim when $f$ is replaced with $f^n$. By the FIC (Proposition~\ref{prop:FIC}), both $f$ and $g_n = k \circ f^n$ represent fully irreducible outer automorphisms $\vphi, \psi_n$ of $F_3$. The ideal Whitehead graph of $\vphi$ is a union of three triangles, while $IW(\psi_n)$ is a union of two triangles. Hence $\vphi$ is principal, while $\psi_n$ is not.
\end{proof}

Theorem~B immediately implies:

\begin{corC}
For $r=3$, there exist a principal periodic geodesic $\mL\in \mP_r$ in $\os$ and a sequence of nonprincipal periodic geodesics $\{\mL_n\}_{n=1}^\infty \subseteq B\mP_r - \mP_r$ such that $\displaystyle \lim_{n \to \infty} \mL_n = \mL$.
\end{corC}

\bibliographystyle{alpha}

\bibliography{PaperRefs}

\begin{thebibliography}{GJLL98}

\bibitem[AK12]{yak}
Y.~Algom-Kfir.
\newblock The metric completion of {O}uter {S}pace.
\newblock {\em Arxiv preprint arXiv:1202.6392}, 2012.

\bibitem[BFH00]{bfh00}
M.~Bestvina, M.~Feighn, and M.~Handel.
\newblock The {T}its {A}lternative for {O}ut $({F}_n)$ {I}: {D}ynamics of
  exponentially-growing automorphisms.
\newblock {\em Annals of Mathematics-Second Series}, 151(2):517--624, 2000.

\bibitem[BH92]{bh92}
M.~Bestvina and M.~Handel.
\newblock Train tracks and automorphisms of free groups.
\newblock {\em The Annals of Mathematics}, 135(1):1--51, 1992.

\bibitem[CL15]{cl15}
T.~Coulbois and M.~Lustig.
\newblock Index realization for automorphisms of free groups.
\newblock {\em ArXiv preprint arXiv:1506.04536}, 2015.

\bibitem[FH11]{fh11}
M.~Feighn and M.~Handel.
\newblock The {R}ecognition {T}heorem for {${\rm Out}(F_n)$}.
\newblock {\em Groups Geom. Dyn.}, 5(1):39--106, 2011.

\bibitem[FM11]{fm11}
S.~Francaviglia and A.~Martino.
\newblock Metric properties of outer space.
\newblock {\em Publicacions Matem{\`a}tiques}, 55(2):433--473, 2011.

\bibitem[GJLL98]{gjll}
D.~Gaboriau, A.~Jaeger, G.~Levitt, and M.~Lustig.
\newblock An index for counting fixed points of automorphisms of free groups.
\newblock {\em Duke mathematical journal}, 93(3):425--452, 1998.

\bibitem[GM16]{gm16}
V.~Gadre and J.~Maher.
\newblock The stratum of random mapping classes.
\newblock {\em ArXiv preprint arXiv:1607.01281}, 2016.

\bibitem[HM11]{hm11}
M.~Handel and L.~Mosher.
\newblock {\em Axes in {O}uter {S}pace}.
\newblock Number 1004. Amer Mathematical Society, 2011.

\bibitem[Kap14]{k14}
I.~Kapovich.
\newblock Algorithmic detectability of iwip automorphisms.
\newblock {\em Bull. Lond. Math. Soc.}, 46(2):279--290, 2014.

\bibitem[KM96]{km96}
V.~A. Kaimanovich and H.~Masur.
\newblock The {P}oisson boundary of the mapping class group.
\newblock {\em Invent. Math.}, 125(2):221--264, 1996.

\bibitem[KP15]{kp15}
I.~Kapovich and C.~Pfaff.
\newblock A train track directed random walk on {${\rm Out}(F_r)$}.
\newblock {\em International {J}ournal of {A}lgebra and {C}omputation},
  25(5):745--798, August 2015.

\bibitem[Mah11]{m11}
J.~Maher.
\newblock Random walks on the mapping class group.
\newblock {\em Duke Mathematical Journal}, 156(3):429--468, 2011.

\bibitem[MP16]{mp13}
L.~Mosher and C.~Pfaff.
\newblock Lone {A}xes in {O}uter {S}pace.
\newblock {\em Algebraic \& Geometric Topology}, 16-6:3385--3418, 2016.

\bibitem[MT14]{mt14}
J.~Maher and G.~Tiozzo.
\newblock Random walks on weakly hyperbolic groups.
\newblock {\em ArXiv preprint arXiv:1410.4173}, 2014.

\bibitem[NPR14]{npr14}
H.~Namazi, A.~Pettet, and P.~Reynolds.
\newblock Ergodic decompositions for folding and unfolding paths in {O}uter
  space.
\newblock {\em ArXiv preprint arXiv:1410.8870}, 2014.

\bibitem[Pfa12]{p12a}
C.~Pfaff.
\newblock {\em Constructing and {C}lassifying {F}ully {I}rreducible {O}uter
  {A}utomorphisms of {F}ree {G}roups}.
\newblock PhD thesis, Rutgers University, 2012.

\bibitem[Pfa13]{IWGII}
C.~Pfaff.
\newblock Ideal {W}hitehead graphs in {${\rm Out}(F_r)$} {II}: the complete
  graph in each rank.
\newblock {\em Journal of Homotopy and Related Structures}, 10(2):275--301,
  2013.

\bibitem[Riv08]{riv08}
I.~Rivin.
\newblock Walks on groups, counting reducible matrices, polynomials, and
  surface and free group automorphisms.
\newblock {\em Duke Math. J.}, 142(2):353--379, 2008.

\bibitem[Sko89]{s89}
R.~Skora.
\newblock Deformations of length functions in groups, preprint.
\newblock {\em Columbia University}, 1989.

\bibitem[Sta83]{s83}
J.R. Stallings.
\newblock Topology of finite graphs.
\newblock {\em Inventiones Mathematicae}, 71(3):551--565, 1983.

\bibitem[TT16]{tt16}
S.~J. Taylor and G.~Tiozzo.
\newblock Random extensions of free groups and surface groups are hyperbolic.
\newblock {\em Int. Math. Res. Not. IMRN}, (1):294--310, 2016.

\bibitem[Vog02]{v02}
K.~Vogtmann.
\newblock Automorphisms of free groups and outer space.
\newblock {\em Geometriae Dedicata}, 94(1):1--31, 2002.

\end{thebibliography}

\end{document}